\documentclass[11pt]{article}
\usepackage{bbm}
\usepackage{psfrag,epic,eepic,epsfig}
\usepackage{fullpage,color}
\usepackage[applemac]{inputenc} 
\usepackage{amsmath,amsfonts,amssymb,amsthm,mathrsfs}
\usepackage[a4paper,vmargin={3.5cm,3.5cm},hmargin={2.5cm,2.5cm}]{geometry}
\usepackage[font=sf, labelfont={sf,bf}, margin=1cm]{caption}
\usepackage{graphicx,graphics}
\usepackage{latexsym}
\usepackage{ae,aecompl}
\usepackage[english]{babel}
 \usepackage[colorlinks=true]{hyperref}
\usepackage{enumerate}
\usepackage[normalem]{ulem}
\usepackage{xcolor}
\usepackage{empheq}
\usepackage[framemethod=tikz]{mdframed}
\usepackage{lipsum}
\usepackage{pifont}
\usepackage{fancybox} 
\usepackage{textcomp}
\usepackage{tikz}
\newcommand*\circled[1]{\tikz[baseline=(char.base)]{
            \node[shape=circle,draw,inner sep=2pt] (char) {#1};}}

\definecolor{gris25}{gray}{0.75}

\definecolor{mycolor}{rgb}{0, 0, 0.1}

\newmdenv[innerlinewidth=0.5pt, roundcorner=4pt,linecolor=mycolor,innerleftmargin=6pt,
innerrightmargin=6pt,innertopmargin=6pt,innerbottommargin=6pt]{mybox}

\renewcommand{\subset}{\subseteq}

\newtheorem{theorem}{Theorem}[section]

\newtheorem{lem}[theorem]{Lemma}
\newtheorem{prop}[theorem]{Proposition}

\date{}

\title{\bf{\textsc{Asymptotics of heights in random trees constructed by aggregation}}}

\author{\Large Bénédicte Haas\thanks{Université Paris 13, Sorbonne Paris Cité, LAGA, CNRS (UMR  7539) 93430 Villetaneuse, France. \newline E--mail:  \textcolor{red}{haas@math.univ-paris13.fr}. \newline This work is partially supported by the ANR GRAAL ANR--14--CE25--0014.}}

\begin{document}

\maketitle

\abstract{To each sequence $(a_n)$ of positive real numbers we associate a growing sequence $(T_n)$ of continuous trees built recursively by gluing at step $n$ a segment of length $a_n$ on a uniform point of the pre--existing tree, starting from a segment  $T_1$ of length $a_1$. Previous works \cite{ADGO14,CH15+} on that model focus on the influence of $(a_n)$ on the compactness and Hausdorff dimension of the limiting tree. Here we consider the cases where the sequence $(a_n)$ is regularly varying with a non--negative index, so that the sequence $(T_n)$ exploses. We  determine the asymptotics of the height of $T_n$ and of the subtrees of $T_n$ spanned by the root and $\ell$ points picked uniformly at random and independently in $T_n$, for all $\ell \in \mathbb N$.}

\bigskip

\tableofcontents

\bigskip

\setlength{\parindent}{0pt}
\setlength{\parskip}{8pt} 

\newpage

\section{Introduction}

A well--known construction of the Brownian continuum random tree presented by Aldous in the first of his series of papers \cite{Ald91a,Ald91,Ald93}  holds as follows. Consider a Poisson point process on $\mathbb R_+$ with intensity $t\mathrm dt$ and ``break" the half--line $\mathbb R_+$ at each point of the process. This gives an ordered sequence of closed segments  with random lengths. Take the first segment and glue on it the second segment at a point chosen uniformly at random (i.e. according to the normalized length measure). Then consider the continuous tree formed by these two first segments and glue 
on it the third segment at a point chosen uniformly at random. And so on. This gluing procedure, called the \textit{line--breaking construction} by Aldous \cite{Ald91a}, gives in the limit a version of the  Brownian CRT. 

We are interested in a generalization of this construction, starting from any sequence of positive terms
$$
(a_n, n \geq 1).
$$
For each $n$, we let $\mathsf b_n$ denote a closed segment of length $a_n$. The construction process then holds as above: we start with  $T_1:=\mathsf b_1$  and then recursively glue the segment $\mathsf b_n$ on a point chosen uniformly  on $T_{n-1}$, for all $n \geq 1$.  The trees $T_n$ are viewed as metric spaces, once endowed with their length metrics, which will be noted $d$ in all cases. This yields in the limit a random real tree obtained as the completion of the increasing union of the trees $T_n$, 
$$
\mathcal T:=\overline{\cup_{n\geq 1}T_n}
$$
that may possibly be infinite.  We let $d$ denote its metric as well, and decide to \emph{root} this tree at one of the two extremities of $\mathsf b_1$. 

This model has been recently studied by Curien and Haas \cite{CH15+} and Amini et al. \cite{ADGO14}. The paper \cite{CH15+} gives necessary and sufficient conditions on the sequence $(a_n)$ for $\mathcal T$ to be compact (equivalently bounded) and studies its Hausdorff dimension. Typically, if 
$$a_n\leq n^{\alpha+\circ(1)}   \text{ and }  a_1+\ldots + a_n= n^{\alpha+\circ(1)} \quad \text{for some }\alpha<0,$$ then almost surely the tree $\mathcal T$ is compact and its set of leaves has Hausdorff dimension $1/|\alpha|$, which ensures that the tree itself has Hausdorff dimension $\max(1,1/|\alpha|)$. This, as an example, retrieves the compactness of the Brownian CRT and that its Hausdorff dimension is 2.  On the other hand, the tree $\mathcal T$ is almost surely unbounded has soon as the sequence $(a_n)$ does not converge to 0.  The issue of finding an exact condition on $(a_n)$ for $\mathcal T$ to be bounded is still open. 
However, Amini et al. \cite{ADGO14} obtained an exact condition for $\mathcal T$ to be bounded, provided that  $(a_n)$ is \emph{non--increasing}. In that case,  almost surely, 
$$\mathcal T \ \text{ is bounded if and only if }  \  \sum_{i\geq 1} i^{-1}a_i<\infty.$$ 
There are also related works, with different gluing rules. Sénizergues \cite{Delphin16} studies a generalization of this $(a_n)$--model where the segments are replaced by $d$--dimensional independent random metric measured spaces ($d \in (0,\infty)$) and the gluing rules depend both on the diameters and the measures of the metric spaces. He shows an unexpected and intriguing Hausdorff dimension. 
In another direction,  Goldschmidt and Haas \cite{GH15} propose a construction of the stable Lévy trees  introduced by Duquesne, Le Gall and Le Jan \cite{DLG02,LGLJ98} that generalizes the  line--breaking construction of the Brownian CRT to this class of trees. Except in the Brownian case,  the stable Lévy trees are not binary and the gluing procedure is then slightly more complex.

The aim of the present paper is to examine the cases where  the $(a_n)$--model obviously leads to an infinite tree and 
we will almost always assume that
$$
\text{the sequence }(a_n) \text{ is regularly varying with index }\alpha \geq 0.
$$
We recall that this means that for all $c>0$,
$$
\frac{a_{\lfloor cn \rfloor}}{a_n} \underset{n \rightarrow \infty} \longrightarrow c^{\alpha},
$$
the prototype example being the power sequence $(n^{\alpha})$.
We refer to Bingham et al.  \cite{BGT89} for background on that topic. 
Our goal is to understand how the tree $T_n$ then grows as $n \rightarrow \infty$.  In that aim, we will study the asymptotic behavior of the height of a typical point of $T_n$ and of the height of $T_n$. We will see that in general these heights do not grow at the same rate.  We will also complete the study of the height of a typical point first by providing a functional convergence, and second by studying  the behavior of the subtrees of $T_n$ spanned by the root and $\ell$ points picked uniformly at random and independently in $T_n$, for all $\ell \in \mathbb N$.

\bigskip

\noindent \textbf{Height of a typical point and height of $T_n$.}
We are interested in the asymptotic behavior of the following quantities:
\begin{enumerate}
\item[$\bullet$] $D_n$: height of a typical point, i.e. given $T_n$, we pick $X_n \in T_n$ uniformly at random in $T_n$ and let $$D_n=d(X_n,\mathsf{root})$$ be its distance to the root;
\item[$\bullet$]  the height of the tree: $$H_n=\max_{v \in T_n} d(v,\mathsf{root}).$$
\end{enumerate}

\medskip

In the particular case where all the lengths $a_n$ are identical, the sequence $(T_n)$ can be coupled with a growing sequence of uniform recursive trees with i.i.d. uniform $(0,1)$ lengths on their edges. This is explained in Section \ref{Sec1}. The asymptotic behavior of the height of a uniform vertex and the height of a random recursive tree \emph{without} edge lengths (i.e. endowed with the graph distance) are well--known, \cite{Devroye87, Pittel94}. From this and the strong law of large number, we immediately get the asymptotic of $D_n$. The behavior of $H_n$ is less obvious. However, Broutin and Devroye  \cite{BD06} develop the material to study the height of random recursive trees with i.i.d. edge lengths, using the underlying branching structure and large deviation techniques. From this, we will deduce that:

\begin{theorem}
\label{thm:BroutinDevroye}
If $a_n=1$ for all $n\geq 1$, 
$$
\frac{D_n}{\ln(n)} \ \overset{\mathbb P}{\underset{n \rightarrow \infty} \longrightarrow} \  \frac{1}{2} \quad \quad \text{and} \quad \quad  \frac{H_n}{\ln(n)} \ \overset{\mathbb P}{\underset {n \rightarrow \infty}\longrightarrow} \ \frac{e^{\beta^*}}{2\beta^*},
$$
where $\beta^*$ is the unique solution in $(0,\infty)$ to the equation $2(e^{\beta}-1)=\beta e^{\beta}$. Approximately,
$
\beta^* \sim 1,594$ and $e^{\beta^*}/2\beta^* \sim 1,544.
$
\end{theorem}

This will be carried out in Section \ref{Sec1}.
Our main contribution yet concerns the cases where the index of regular variation $\alpha$ is strictly positive. In that case we introduce  a random variable $\xi_{(\alpha)}$ characterized by its Laplace transform $\mathbb E[\lambda \xi_{(\alpha)}]=\exp(\phi_{(\alpha)}(\lambda))$, $\lambda \in \mathbb R$ where 
\begin{eqnarray}
\label{TransfoLapl}
\phi_{(\alpha)}(\lambda) = \frac{\alpha+1}{\alpha}\int_0^1 \big(\exp(\lambda u)-1 \big)\frac{1-u}{u} \mathrm du 
  =  \frac{\alpha+1}{\alpha}\sum_{k \geq 1}\frac{\lambda^k}{(k+1)!k}.
\end{eqnarray}
The Lévy--Khintchine formula ensures that $\xi_{(\alpha)}$ is infinitely divisible. Note also that $\xi_{(\alpha)}$ is stochastically decreasing in $\alpha$. Our main result is:

\medskip

\begin{theorem}
\label{thm:heightmax}
Assume that $(a_n)$ is regularly varying with index $\alpha>0$. Then,
\begin{enumerate}
\vspace{-0.3cm}
\item[\emph{(i)}] 
$$
\frac{D_n}{a_n} \ \overset{\mathrm{(d)}}{\underset{n \rightarrow \infty}\longrightarrow} \ \xi_{(\alpha)} 
$$
\item[\emph{(ii)}] 
$$
 \frac{H_n \cdot \ln(\ln(n))}{a_n \ln(n)} \ \overset{\mathrm{a.s.}}{\underset{n \rightarrow \infty}\longrightarrow} \ 1.
$$
\end{enumerate}
More precisely, in \emph{(i)}, $\mathbb E[\exp(\lambda a_n^{-1}D_n)]$ converges to $\mathbb E[\exp(\lambda \xi_{\alpha})]$ for all $\lambda \in \mathbb R$, which in particular implies the convergence of all positive moments.
\end{theorem}

\medskip

The proof of (i)  is undertaken in Section \ref{SecOneDim} and relies on the powerful observation from \cite{CH15+} that $D_n$ can be written as the sum of i.i.d. random variables.  The proof of (ii), and in particular of the lower bound, is more intricate. It relies on the 
\emph{second moment method} and requires to get the joint distribution of the paths from the root to two points marked independently, uniformly in the tree $T_n$ (established in Section \ref{sec:Marking2})
as well as precise deviations bounds for the convergence (i) (established in Section \ref{SecOneDim}). The core of the proof of (ii) is undertaken in Section \ref{SecHeight}.

\smallskip

The two previous statements on the asymptotic behavior of $D_n$ can actually be grouped together and slightly generalized as follows:

\begin{prop}
\label{cvTypHeight}
Assume that $(a_n)$ is regularly varying with index $\alpha \geq 0$. Then, 
$$
\frac{D_n}{\sum_{i=1}^n i^{-1}a_i} \underset{n \rightarrow \infty} \longrightarrow 
\left\{\begin{array}{ll} \vspace{0.2cm} \displaystyle \alpha \xi_{(\alpha)} & \text{ if $\alpha>0$ \hspace{0.1cm} \emph{(}convergence in distribution\emph{)}} 
\\ \vspace{0.2cm} \displaystyle \frac{1}{2} & \text{ if $\alpha=0$ and $\sum_{i=1}^{\infty} i^{-1}a_i=\infty$ \hspace{0.1cm} \emph{(}convergence in probability\emph{)}} \\ D_{\infty} &  \text{ if $\alpha=0$ and $\sum_{i=1}^{\infty} i^{-1}a_i<\infty$ \hspace{0.1cm} \emph{(}convergence in distribution\emph{)}}\end{array}\right.
$$
where $D_{\infty}$ denotes a positive random variable with finite expectation.
\end{prop}

This will be explained in the remark around (\ref{TypHGen}) in Section \ref{SecOneDim}.

\bigskip

\smallskip

\textbf{Height of the $n$--th leaf and height of a uniform leaf.} In the recursive construction of $(T_n)$, we can label the leaves $L_1,L_2,\ldots$ by order of apparition, so that the leaf $L_n$ belongs to the segment $\mathsf b_n$. We then let $L_{n,\star}$ denote a leaf chosen uniformly at random amongst the $n$ leaves of $T_n$.
Theorem \ref{thm:heightmax} (i) implies that when $(a_n)$ varies regularly with index $\alpha>0$,
\begin{equation}
\label{UnifLeaf}
\frac{d(L_n,\mathsf{root})}{a_n} \ \overset{\mathrm{(d)}}{\underset{n \rightarrow \infty}\longrightarrow} \ 1+\xi_{(\alpha)} \quad \quad \text{and} \quad \quad  \frac{d(L_{n,\star},\mathsf{root})}{a_n} \ \overset{\mathrm{(d)}}{\underset{n \rightarrow \infty}\longrightarrow} \ (1+\xi_{(\alpha)}) U^{\alpha}, 
\end{equation}
where $U$ is uniform on $(0,1)$, independent of $\xi_{(\alpha)}$.
The first convergence is simply due to the fact that the distance $d(L_n,\mathsf{root})$ is distributed as $a_n+D_{n-1}$, since the segment $\mathsf b_n$ is inserted on a uniform point of $T_{n-1}$. 
The second convergence  is explained in Section \ref{SecOneDim}.
When $a_n=n$ for all $n$, $d(L_n,\mathsf{root})$ and $d(L_{n,\star},\mathsf{root})$ both divided by $\ln(n)$ converge to $1/2$, almost surely and in probability respectively (see Section \ref{Sec1}).

\bigskip

\smallskip

\noindent \textbf{Functional convergence.} The convergence of the height of a typical point can actually be improved into a functional convergence when the index of regular variation is strictly positive.  As above, let $X_n$ be a point picked uniformly in $T_n$ and for each positive integer $k \leq n$, let $X_n(k)$ denote its projection onto $T_k$. Let then $$D_n(k):=d\left(X_n(k),\mathsf{root}\right), \quad 1 \leq k \leq n$$ be the non--decreasing sequence of  distances of these branch--points to the root. If a climber decides to climb from the root to the typical point $X_n$ at speed 1, $D_n(k)$ is the time he will spend in $T_k$. The proof of Theorem \ref{thm:heightmax} (i) can be adapted to get the behavior as $n \rightarrow \infty$ of the sequence $(D_n(k),1 \leq k \leq n)$. 
To do so, introduce for $\alpha>0$ the càdlàg Markov process with independent, positive increments defined by
\begin{equation}
\label{processXi}
\xi_{(\alpha)}(t):=\sum_{t_i \leq t} v_i, \quad  t \geq 0,
\end{equation}
where $(t_i,v_i)$  is a Poisson point process with intensity $(\alpha+1)t^{-\alpha-1} \mathbbm 1_{\{v \leq t^{\alpha}\}} \mathrm dt \mathrm dv $ on $(0,\infty)^2$. (We note that $\xi_{(\alpha)}(1)$ is distributed as the r.v. $\xi_{(\alpha)}$ defined via (\ref{TransfoLapl}).) This process is $\alpha$--self-similar, in the sense that for all $a>0$,
$$
\big(\xi_{(\alpha)}(at), t \geq 0 \big) \ \overset{\mathrm{(d)}}{=} \  \big(a^{\alpha}\xi_{(\alpha)}(t), t \geq 0 \big).
$$ 

\medskip

\begin{prop}
\label{propfunctional}
If $(a_n)$ is regularly varying with index $\alpha>0$, 

\vspace{-0.1cm}
$$
\left(\frac{D_n\left(\lfloor nt \rfloor\right)}{a_n}, 0 \leq t \leq 1 \right) \ \overset{\mathrm{(d)}}{\underset{n \rightarrow \infty}\longrightarrow} \ \left(\xi_{(\alpha)}(t), 0 \leq t \leq 1\right)
$$

for the Skorokhod topology on $D([0,1],\mathbb R_+)$, the set of càdlàg functions from $[0,1]$ to $\mathbb R_+$.
\end{prop}

This  is proved in Section \ref{Secfunc}. 

\bigskip

\smallskip

\noindent \textbf{Gromov--Prohorov--type convergence.} Last, fix $\ell$ a positive integer, and given $T_n$, let $X_n^{(1)},\ldots,X_n^{(\ell)}$ be $\ell$ points picked independently and uniformly at random in $T_n$. Our goal is to describe the asymptotic behavior of $T_n(\ell)$, the subtree of $T_n$ spanned by these $\ell$ marked points and the root. 
In that aim,  for all $1 \leq i,j \leq \ell$, we denote by $\mathsf B_n^{(i,j)}$ the  point in $T_n(\ell)$ at which the paths from the root to $X_n^{(i)}$ and  from the root to $X_n^{(j)}$ separate, with the convention that $\mathsf B_n^{(i,j)}=X_n^{(i)}$ when $X_n^{(i)}$ belongs to the path  from the root to $X_n^{(j)}$. For regularly varying sequences of lengths $(a_n)$, the tree $T_n(\ell)$ appropriately rescaled  converges to 
a ``star--tree" with $\ell$ branches with random i.i.d. lengths. More precisely:
\begin{prop}
\label{prop:Prokho}
\emph{(i)} Assume that $(a_n)$ is regularly varying with index $\alpha > 0$. Then,

\vspace{-0.1cm}
$$
\left(\left(\frac{d\big(X_n^{(i)}, \mathsf{root} \big)}{a_n}, 1 \leq i \leq \ell\right),\frac{\max_{1\leq i \neq j \leq \ell}d\big(\mathsf B_n^{(i,j)},\mathsf{root}\big)}{a_n} \right) \underset{n \rightarrow \infty}{\overset{\mathrm{(d)}}\longrightarrow} \left(\left(\xi^{(i)}_{(\alpha)},1 \leq i \leq \ell\right),0 \right) 
$$ 

where $\xi^{(1)}_{(\alpha)},\ldots,\xi^{(\ell)}_{(\alpha)}$ are i.i.d. with distribution \emph{(\ref{TransfoLapl})}.

\smallskip

\emph{(ii)} Assume that $(a_n)$ is regularly varying with index $0$ and that $\sum_{i=1}^{\infty}i^{-1}a_i=\infty$. Then, 

\vspace{-0.1cm}
$$
\left(\left(\frac{d\big(X_n^{(i)}, \mathsf{root} \big)}{\sum_{i=1}^n i^{-1} a_i}, 1 \leq i \leq \ell\right),\frac{\max_{1\leq i \neq j \leq \ell}d\big(\mathsf B_n^{(i,j)},\mathsf{root}\big)}{a_n} \right) \underset{n \rightarrow \infty}{\overset{\mathrm{\mathbb P}}\longrightarrow} \left(\left(\frac{1}{2},\ldots, \frac{1}{2}\right),0 \right).
$$ 
\end{prop}

\bigskip

\bigskip

\textbf{Notation.} Throughout the paper, we use the notation
$$
A_n:=\sum_{i=1}^n a_i, \quad \text{for all }n \in \mathbb N.
$$

\section{Height of a typical point}

Fix $n$, and
given $T_n$, let $X_n$ be a point picked uniformly on $T_n$. 
The goal of this section is to settle different results on the distribution of the distance of this marked point to the root, mainly when the sequence $(a_n)$ is regularly varying with a  strictly positive index. Our approach entirely relies on the fact that this distance can be written as the sum of independent, non--negative random variables. More precisely, as noticed in \cite{CH15+}, 
the distances $D_n(k)$ to the root of the projections of $X_n$ onto $T_k, k \leq n$ can jointly be written in the following form: 
\begin{equation}
\label{D_n(k)}
D_n(k)=\sum_{i=1}^k a_i V_i \mathbbm 1_{\big\{U_i \leq \frac{a_i}{A_i}\big\}}, \quad \forall k \leq n,
\end{equation}
where $U_i,V_i,1 \leq i \leq n$ are all uniformly distributed on $(0,1)$ and independent. In particular,
the distance $D_n$ of $X_n$ to the root writes
\begin{equation}
\label{D_n}
D_n=\sum_{i=1}^n a_i V_i \mathbbm 1_{\left\{U_i \leq \frac{a_i}{A_i}\right\}}. 
\end{equation}
To see this, we roughly proceed as follows. Consider the projection $X_n(n-1)$ of $X_n$ onto $T_{n-1}$. By construction, it is uniformly distributed on $T_{n-1}$ given $T_{n-1}$, and then:
\begin{enumerate}
\item[$\bullet$] either $X_n \in T_{n-1}$ and $X_n(n-1)=X_n$, which occurs with probability $A_{n-1}/A_n$, 
\item[$\bullet$]  or $X_n \in T_n \backslash T_{n-1}$ and $d(X_n,X_n(n-1))=a_n V_n$ with $V_n$ uniform on $(0,1)$ and independent of $T_{n-1}$, which occurs with probability $a_n/A_n$. 
\end{enumerate}
Iterating this argument gives (\ref{D_n(k)}). An obvious consequence is that 
\begin{equation}
\label{espDn}
\mathbb E\left[D_n\right]=\frac{1}{2} \cdot \sum_{i=1}^n\frac{a_i^2}{A_i}.
\end{equation}
The rest of this section is organized as follows. In Section \ref{SecRV} we start by recalling some classical bounds for regularly varying sequences that will be used throughout the paper. The first part of Section \ref{SecOneDim} concerns the asymptotic behavior of the height $D_n$, with the proofs of Theorem \ref{thm:heightmax} (i) and its corollaries (\ref{UnifLeaf}), as well as Proposition \ref{cvTypHeight}. The second part of Section \ref{SecOneDim} is devoted to  the implementation of bounds (Lemma \ref{lem:equiv}) that will be crucial for the proof of Theorem \ref{thm:heightmax} (ii) on the behavior of the height of $T_n$, proof that will be undertaken in Section \ref{SecHeight}. Last, Section \ref{Secfunc} contains the proof of Proposition \ref{propfunctional}. 

\subsection{Bounds for regularly varying sequences}
\label{SecRV}

Assume that  $(a_n)$ is regularly varying with index $\alpha \geq 0$. We recall some classical bounds that will be useful at different places in the paper.

Fix $\varepsilon>0$. From \cite[Theorem 1.5.6 and Theorem 1.5.11]{BGT89}, there exists an integer $i_{\varepsilon}$ such that for all $n \geq i \geq i_{\varepsilon}$,
\begin{equation}
\label{regularvar1}
(1-\varepsilon)\left(\frac{i}{n}\right)^{\alpha+\varepsilon} \leq \frac{a_i}{a_n} \leq (1+\varepsilon)\left(\frac{i}{n}\right)^{\alpha-\varepsilon}
\end{equation}
and
\begin{equation}
\label{regularvar2}
\frac{(1- \varepsilon)(\alpha+1)}{i} \leq \frac{a_i}{A_i} \leq \frac{(1+\varepsilon)(\alpha+1)}{i}
\end{equation}
(\cite[Theorem 1.5.6]{BGT89} and \cite[Theorem 1.5.11]{BGT89} are stated for regularly varying functions, but can be used for regularly varying sequences, using that $f(x):=a_{\lfloor x \rfloor}$ is a  varying regularly function).

Moreover, still by  \cite[Theorem 1.5.11]{BGT89}, 
\begin{equation}
\label{regularvar3}
\frac{a_n}{\sum_{i=1}^n i^{-1}a_i } \underset{n \rightarrow \infty} \longrightarrow \alpha.
\end{equation}

\subsection{One dimensional convergence and deviations}
\label{SecOneDim}

For $\alpha>0$, recall the definition of the random variable $\xi_{(\alpha)}$ defined via its Laplace transform $$\mathbb E[\lambda \xi_{(\alpha)}]=\exp(\phi_{(\alpha)}(\lambda)), \quad \lambda \in \mathbb R$$ with $\phi_{(\alpha)}$ given by (\ref{TransfoLapl}).
With the expression (\ref{D_n}), it is easy to find the asymptotic behavior of $(D_n)$ by computing its Laplace transform and then get Theorem \ref{thm:heightmax} (i). We more precisely have:

\smallskip

\begin{lem}
\label{lemLaplace}
Assume that $(a_n)$ is regularly varying with index $\alpha>0$. Then,
\vspace{-0.3cm}
\begin{enumerate}
\item[\emph{(i)}]  For all $\lambda \in \mathbb R$
$$\mathbb E\left[\exp\left(\lambda \frac{D_n}{a_n}\right)\right] \underset{n \rightarrow \infty} \longrightarrow \exp(\phi_{(\alpha)}(\lambda)).$$
\item[\emph{(ii)}] For all $c>1$, there exists $n_{c}$ such that for all  $n \geq n_{c}$ and all $\lambda \geq 0$,
$$
\mathbb E\left[\exp\left(\lambda \frac{D_n}{a_n}\right) \right] \leq \exp\left(c(1+\alpha^{-1})\lambda \exp\left(c\lambda \right) \right).
$$
\end{enumerate}
\end{lem}

\medskip

\textbf{Proof.} (i) For all $\lambda \neq 0$, we get from (\ref{D_n}) that
\begin{eqnarray*}
\mathbb E\left[\exp\left(\lambda \frac{D_n}{a_n}\right) \right]&=&\prod_{i=1}^n \mathbb E \left[ \exp\left(\lambda \frac{a_i}{a_n} V \mathbbm 1_{\left\{U \leq \frac{a_i}{A_i} \right\}} \right) \right] \\
&=& \prod_{i=1}^n \left( 1- \frac{a_i}{A_i}+ \frac{a_i}{A_i} \frac{\left(\exp\left(\lambda \frac{a_i}{a_n}\right)-1 \right)}{\lambda \frac{a_i}{a_n}}\right)
\end{eqnarray*}
where $U,V$ are uniform on $(0,1)$ and independent (if $a_i=0$ for some $i$ we use the convention \linebreak $(\exp(0)-1)/0=1$). Now assume that $\lambda>0$ (the following lines hold similarly for $\lambda<0$ by adapting the bounds).  Using (\ref{regularvar1}), (\ref{regularvar2}) together with the fact that $\ln(1+x) \sim x $ as $x \rightarrow 0$ and that  $x \mapsto x^{-1}(\exp(x)-1)$ is increasing on $(0,\infty)$ and converges to 1 as $x \rightarrow 0$,  leads to the existence of an integer $j_{\varepsilon}$ such that for $n \geq j_{\varepsilon}$
\begin{eqnarray}
\label{ub1}
&&\nonumber c_{\varepsilon}(n) + (1-\varepsilon )^2(\alpha+1) \sum_{i=j_{\varepsilon}}^n  \frac{1}{i} \left(\left(\frac{\exp\left(\lambda (1-\varepsilon) \left(\frac{i}{n}\right)^{\alpha+\varepsilon}\right)-1}{\lambda  (1-\varepsilon)\left(\frac{i}{n}\right)^{\alpha+\varepsilon}}\right)-1\right) \\
&\leq &\nonumber \ln \left( \mathbb E\left[\exp\left(\lambda \frac{D_n}{a_n}\right) \right]\right) \\
&\leq & c_{\varepsilon}(n)  + (1+\varepsilon)^2(\alpha+1) \sum_{i=j_{\varepsilon}}^n \frac{1}{i} \left(\left(\frac{\exp\left(\lambda (1+\varepsilon) \left(\frac{i}{n}\right)^{\alpha-\varepsilon}\right)-1}{\lambda (1+\varepsilon)\left(\frac{i}{n}\right)^{\alpha-\varepsilon}}\right)-1 \right),
\end{eqnarray}
where 
$$
c_{\varepsilon}(n):= \sum_{i=1}^{j_{\varepsilon}-1} \ln \left(1- \frac{a_i}{A_i}+ \frac{a_i}{A_i} \frac{\left(\exp\left(\lambda \frac{a_i}{a_n}\right)-1 \right)}{\lambda \frac{a_i}{a_n}} \right) \underset{n \rightarrow \infty}\longrightarrow 0
$$
since $a_n \rightarrow \infty$.
Writing $\frac{1}{i}=\frac{1}{n}\times\frac{n}{i}$, we recognize  Riemann sums in the lower and upper bounds, which,
letting first $n \uparrow \infty$ and then $\varepsilon \downarrow 0$ gives
$$
\ln \left( \mathbb E\left[\exp\left(\lambda \frac{D_n}{a_n}\right) \right]\right) \underset{n \rightarrow \infty}\longrightarrow (\alpha+1) \int_0^1 \frac{1}{x}\left(\left(\frac{\exp(\lambda x^{\alpha})-1}{\lambda x^{\alpha}} \right)-1\right) \mathrm dx=:\phi_{(\alpha)}(\lambda).
$$
It is easy to see with the change of variables $y=x^{\alpha}$ in the integral and then the power series expansion of the exponential function that this expression of $\phi_{(\alpha)}(\lambda)$ indeed corresponds to (\ref{TransfoLapl}).

(ii). Fix $c>1$. Using the upper bound (\ref{ub1}) and the fact that
$$
\frac{\exp(x)-1}{x}-1 \leq x \exp(x) \quad \text{for all }x> 0
$$
we see that for all $0<\eta<\alpha$ and then for all $n$ large enough
\begin{eqnarray*}
 \ln \left( \mathbb E\left[\exp\left(\lambda \frac{D_n}{a_n}\right) \right]\right) 
 \leq  c_{\eta}(n) + (1+\eta)^3(\alpha+1) \lambda\exp(\lambda(1+\eta))   \sum_{i=j_{\eta}}^n  \frac{1}{i} \left(\frac{i}{n}\right)^{\alpha-\eta}.
 \end{eqnarray*}
 Note also, using $\ln(1+x) \leq x$, that 
 $$
 c_{\eta}(n) \leq  \sum_{i=1}^{j_{\eta}-1} \frac{a_i}{A_i} \lambda \frac{a_i}{a_n} \exp\left(\lambda \frac{a_i}{a_n}\right)
 $$
 which, clearly, is smaller than $\eta \lambda \exp(\eta \lambda)$ for $n$ large enough and all $\lambda \geq 0$. Gathering this together, we get that for all $n$ large enough (depending on $\eta$) and all $\lambda \geq 0$,
 $$
  \ln \left( \mathbb E\left[\exp\left(\lambda \frac{D_n}{a_n}\right) \right]\right) \leq \left(\eta+(1+\eta)^4 \right) \frac{1+\alpha}{\alpha-\eta} \lambda \exp\left(\lambda(1+\eta)\right).
 $$
 Taking $\eta$ small enough so that $\left(\eta+(1+\eta)^4 \right) \alpha \leq c(\alpha-\eta)$ gives the expected upper bound.
$\hfill \square$

\bigskip

\textbf{Remark (height of a uniform leaf).} We keep the notation of the Introduction and let $L_{n,\star}$ denote a leaf chosen uniformly at random amongst the $n$ leaves of $T_n$. Then the previous result implies that when $(a_n)$ is regularly varying with index $\alpha>0$,
$$\frac{d(L_{n,\star},\mathsf{root})}{a_n} \ \overset{\mathrm{(d)}}{\underset{n \rightarrow \infty}\longrightarrow} \ (1+\xi_{(\alpha)}) U^{\alpha}$$
with $U$ uniformly distributed on $(0,1)$ and independent of $\xi_{(\alpha)}$. To see this, one could use that the distribution of $(1+\xi_{(\alpha)}) U^{\alpha}$ is characterized by its positive moments (since it has exponential moments, since $\xi_{(\alpha)}$ has), together with the fact that for each $p \geq 0$, the  $p$--th moment $\mathbb E[(d(L_{n,\star},\mathsf{root})/a_n)^p]$ converges to $\mathbb E[((1+\xi_{(\alpha)}) U^{\alpha})^p]$. To prove this last convergence, note that
$$
\mathbb E\left[\left( \frac{d(L_{n,\star},\mathsf{root})}{a_n}\right)^p \right]=\frac{1}{n} \sum_{i=1}^n \mathbb E\left[\left( \frac{d(L_{i},\mathsf{root})}{a_i}\right)^p \right]\left(\frac{a_i}{a_n}\right)^p.
$$
Since $d(L_{i},\mathsf{root})-a_i$ is uniformly distributed on $T_{i-1}$ (by construction) we know from the previous lemma that, divided by $a_i$, it converges in distribution to $\xi_{(\alpha)}$, and that more precisely there is convergence of all positive and exponential moments. Together with (\ref{regularvar1}), this leads to the convergence of  $\mathbb E[(d(L_{n,\star},\mathsf{root})/a_n)^p]$ to $\mathbb E[(1+\xi_{(\alpha)})^p]/(\alpha p+1)$, as expected.

\bigskip

\textbf{Remark (other sequences $\boldsymbol{(a_n)}$).} It is easy to adapt Part (i) of the proof to get that for a general sequence $(a_n)$ of positive terms such that  $(\sum_{i=1}^n A_i^{-1}a_i^2)^{-1}\max_{1 \leq i \leq n}a_i \rightarrow 0$ as $n \rightarrow \infty$,
\begin{equation}
\label{TypHGen}
\frac{D_n}{\sum_{i=1}^n A_i^{-1}a_i^2} \overset{\mathbb P}{\underset {n \rightarrow \infty}\longrightarrow} \frac{1}{2}.
\end{equation}
It is easy to check that the above condition on $(a_n)$ holds if $(a_n)$ is regularly varying with index 0 and $\sum_{i=1}^{\infty} i^{-1}a_i=\infty$ (recall  (\ref{regularvar2}),(\ref{regularvar3})), leading in that case to
$$
\frac{D_n}{\sum_{i=1}^n i^{-1}a_i} \overset{\mathbb P}{\underset {n \rightarrow \infty}\longrightarrow}  \frac{1}{2}.
$$
In particular this recovers the first part of Theorem \ref{thm:BroutinDevroye}. To illustrate with other $0$--regularly varying sequences, consider $a_n=(\ln(n))^{\gamma}, \gamma \in \mathbb R$. Then:
$$
\left\{\begin{array}{ccccc} 
\vspace{0.2cm}
\displaystyle \frac{D_n}{(\ln(n))^{\gamma+1}} & \overset{\mathbb P}{\underset{n \rightarrow \infty} \longrightarrow} & \ \displaystyle \frac{1}{2(\gamma+1)} & \text{when } \gamma>-1 \\ 
\vspace{0.1cm}
\displaystyle \frac{D_n}{\ln(\ln(n))} & \overset{\mathbb P}{\underset{n \rightarrow \infty} \longrightarrow} &   \ \displaystyle \frac{1}{2}  & \text{when } \gamma=-1  \\ D_n   & \overset{\mathrm{a.s.}}{\underset{n \rightarrow \infty} \longrightarrow} &   D_{\infty} & \text{when } \gamma<-1, 
\end{array}\right.
$$
where the last line is due to the fact that $(D_n)$ is stochastically increasing (by (\ref{D_n})) and that $\lim_n \mathbb E[D_n]$ is finite when $\gamma<-1$ (by (\ref{espDn}),(\ref{regularvar2})), which implies that $(D_n)$ converges in distribution to a r.v. $D_{\infty}$ with finite expectation. Note that this last argument actually holds for any sequence $(a_n)$ such that $\sum_{i=1}^{\infty} A_i^{-1}a_i^2<\infty$ (which is equivalent to $\sum_{i=1}^{\infty} i^{-1}a_i<\infty$ when $(a_n)$ is regularly varying, necessarily with index 0). All these remarks lead to Proposition \ref{cvTypHeight}, using again (\ref{regularvar3}) when $\alpha>0$.

\bigskip

We come back to the case where $\alpha>0$ and note the following behavior of the maximum of $n$ i.i.d. copies of $\xi_{(\alpha)}$.

\medskip

\begin{prop}
Let $\xi_{(\alpha,1)},\ldots, \xi_{(\alpha,n)}$ be i.i.d. copies of $\xi_{(\alpha)}$. Then,
$$
\frac{\max\left\{\xi_{(\alpha,1)},\ldots, \xi_{(\alpha,n)}\right\} \times \ln(\ln(n))}{\ln(n)} \overset{\mathbb P}{\underset {n \rightarrow \infty}\longrightarrow}1.
$$
\end{prop}

\medskip

\textbf{Proof.}
From (\ref{TransfoLapl}), we know that the random variable $\xi_{(\alpha)}$ is infinitely divisible and the support of its Lévy measure is $[0,1]$. By \cite[Theorem 8.2.3]{BGT89}, this implies that
$$
\exp\left(\lambda x \ln(x)\right) \mathbb P\left(\xi_{(\alpha)}>x\right) \underset{x \rightarrow \infty}\longrightarrow 0 \quad \text{when } \lambda<1 
$$
and
$$
\exp\left(\lambda x \ln(x)\right) \mathbb P\left(\xi_{(\alpha)}>x\right) \underset{x \rightarrow \infty}\longrightarrow \infty \quad \text{when } \lambda>1.
$$
Besides, the independence of the $\xi_{(\alpha,i)},1 \leq i \leq n$ leads to
$$
\ln \left(\mathbb P \left(\max\left\{\xi_{(\alpha,1)},\ldots, \xi_{(\alpha,n)}\right\} \leq u \frac{\ln(n)}{\ln(\ln(n))} \right)\right) \underset{n \rightarrow \infty} \sim -n \mathbb P\left(\xi_{(\alpha)} > u \frac{\ln(n)}{\ln(\ln(n))}\right),
$$
for all $u>0$. With the above estimates, it is straightforward that the right--hand side converges to 0 when $u>1$ and to $-\infty$ when $u<1$.
$\hfill \square$

\bigskip

We will not directly use this result later in the paper, but this may be seen as a hint that the height $H_n$ may be asymptotically proportional to $n^{\alpha}\ln(n)/\ln(\ln(n))$. To prove this rigorously, we will actually use the following estimates. 

\bigskip

\begin{lem}
\label{lem:equiv}
Assume that $(a_n)$ is regularly varying with index $\alpha>0$ and fix $\gamma>0$.
\vspace{-0.3cm}
\begin{enumerate}
\item[\emph{(i)}]
Then for all $\gamma'<\gamma$,
$$
n^{\gamma'} \mathbb P\left(\frac{D_n}{a_n}>\gamma \frac{\ln(n)}{\ln(\ln(n))}\right) \underset{n \rightarrow \infty}\longrightarrow 0 
$$
whereas for all $\gamma'>\gamma$, 
$$
n^{\gamma'} \mathbb P\left(\frac{D_n}{a_n}>\gamma \frac{\ln(n)}{\ln(\ln(n))}\right) \underset{n \rightarrow \infty}\longrightarrow \infty. 
$$

\item[\emph{(ii)}] Fix $c\in (0,1)$. Then for all $\gamma'<\gamma$
$$
n^{\gamma'} \mathbb P\left(\frac{D_n-D_{n}(\lfloor nc \rfloor)}{a_n}>\gamma \frac{\ln(n)}{\ln(\ln(n))}\right) \underset{n \rightarrow \infty}\longrightarrow 0 
$$
whereas for all $\gamma'>\gamma$, 
$$
n^{\gamma'} \mathbb P\left(\frac{D_n-D_{n}(\lfloor nc \rfloor)}{a_n}>\gamma \frac{\ln(n)}{\ln(\ln(n))}\right) \underset{n \rightarrow \infty}\longrightarrow \infty. 
$$
\end{enumerate}
\end{lem}

\textbf{Proof.} Of course, since $D_n-D_{n}(\lfloor nc \rfloor) \leq D_n$, we only need to prove the convergence to 0 in (i) for $\gamma'<\gamma$ and the convergence to $\infty$ in (ii) for $\gamma'>\gamma$.

{(i)} Let $\gamma'<\gamma$ and take $a,d$ such that $a>\gamma'/\gamma,d>1$ and $ad<1$. From the upper bound of Lemma \ref{lemLaplace}
 (ii), we see that for $n$ large enough
\begin{eqnarray*}
n^{\gamma'} \mathbb P \left(\frac{D_n}{a_n} > \gamma \frac{ \ln(n)}{\ln(\ln(n))} \right) &=& n^{\gamma'}  \mathbb P \left(a\ln(\ln(n))\frac{D_n}{a_n} > a \gamma \ln(n) \right) \\
&\leq & \exp(\gamma' \ln(n)) \cdot \mathbb E \left[\exp\left( a \ln(\ln(n))\frac{D_n}{a_n}\right) \right] \cdot  \exp(-a\gamma \ln(n)) \\
&\underset{\mathrm{Lemma \ref{lemLaplace} (ii)}}\leq & \exp\Big( \ln(n) \times (\gamma' -a \gamma) + (1+\alpha^{-1}) ad\ln(\ln(n))(\ln(n))^{ad} \Big)
\end{eqnarray*}
and this converges to 0 since $ad<1$ and $a\gamma>\gamma'$.

{(ii)} Let $\gamma'>\gamma$. We will (stochastically) compare the random variable $a_n^{-1}(D_n-D_n(\lfloor nc \rfloor))$ with a binomial $\mathrm{Bin}(\lfloor a n \rfloor,b/n)$ distribution, with appropriate $a,b>0$. A simple application of Stirling's formula will then lead to the expected result. Recall from (\ref{D_n}) and (\ref{D_n(k)}) that
$$
\frac{D_n-D_n(\lfloor nc\rfloor)}{a_n}=\sum_{i=\lfloor nc \rfloor+1}^n \frac{a_i}{a_n} V_i \mathbbm 1_{\left\{U_i \leq \frac{a_i}{A_i}\right\}},
$$
with $U_i,V_i,i \geq 1$ i.i.d. uniform on $(0,1)$. Then note from (\ref{regularvar1}) and (\ref{regularvar2})  that  for all $\varepsilon,d \in (0,1)$
$$
\sum_{i=\lfloor dn \rfloor+1}^n \frac{a_i}{a_n} V_i \mathbbm 1_{\left\{U_i \leq \frac{a_i}{A_i}\right\}} \geq (1-\varepsilon) d^{\alpha+\varepsilon} \sum_{\lfloor dn \rfloor +1}^n V_i \mathbbm 1_{\left \{U_i \leq \frac{\alpha+1}{2n}\right \}}
$$
provided that $n$ is large enough. Now take $\varepsilon \in (0,1)$ small enough and $d \in (c,1)$ large enough so that $ \gamma<(1-\varepsilon)^2 d^{\alpha+\varepsilon}\gamma'$. Setting $N_{n,\varepsilon,d}:=\sum_{\lfloor dn \rfloor +1}^n \mathbbm1_{\left\{V_i \geq 1-\varepsilon\right\}}$, we have,
\begin{eqnarray}
\nonumber
&& \mathbb P\left(\frac{D_n-D_{n}(\lfloor nc \rfloor)}{a_n}>\gamma \frac{\ln(n)}{\ln(\ln(n))}\right) \\
\nonumber
&\geq &  \mathbb P\left((1-\varepsilon) d^{\alpha+\varepsilon} \sum_{\lfloor dn \rfloor +1}^n V_i \mathbbm 1_{\left \{U_i \leq \frac{\alpha+1}{2n} \right\}}
 >\gamma \frac{\ln(n)}{\ln(\ln(n))}\right) \\
 \nonumber
 &\underset{(U_i) \text{ indep. } (V_i)}\geq & \mathbb P\left((1-\varepsilon)^2 d^{\alpha+\varepsilon} \sum_{i=1}^{\lfloor \varepsilon (1-d)n/2 \rfloor} \mathbbm 1_{\left\{U_i \leq \frac{\alpha+1}{2n}\right\}}
 >\gamma \frac{\ln(n)}{\ln(\ln(n))}, N_{n,\varepsilon,d} \geq \frac{\varepsilon(1-d)  n}{2}  \right) \\
 \label{lowerboundBin}
 &\geq & \mathbb P\left(\mathrm{Bin}\left(\left\lfloor \frac{\varepsilon (1-d)n}{2} \right\rfloor,\frac{\alpha+1}{2n}\right) >\frac{\gamma}{(1-\varepsilon)^2 d^{\alpha+\varepsilon}} \frac{\ln(n)}{\ln(\ln(n))} \right) \\
 \nonumber
 &&\hspace{-0.3cm}-\mathbb P\left( \mathrm{Bin}\left(n-\lfloor dn \rfloor, \varepsilon\right) < \frac{\varepsilon (1-d)n}{2}  \right).
\end{eqnarray}
One the one hand, the theory of large deviations for the binomial distribution gives
$$
\mathbb P\left( \mathrm{Bin}\left(n-\lfloor dn \rfloor, \varepsilon\right) < \varepsilon (1-d)n /2 \right) \leq \exp(-hn), 
$$
with $h>0$. On the other hand, a simple application of Stirling's formula implies that
\begin{equation}
\label{minoBinom}
\mathbb P\left(\mathrm{Bin}\left(\left\lfloor \frac{\varepsilon (1-d)n}{2} \right\rfloor,\frac{\alpha+1}{2n}\right) >\frac{\gamma}{(1-\varepsilon)^2 d^{\alpha+\varepsilon}} \frac{\ln(n)}{\ln(\ln(n))} \right) \geq n^{-\frac{\gamma}{(1-\varepsilon)^2 d^{\alpha+\varepsilon}}+\circ(1)}
\end{equation}
(this is well--known, a proof is given below). 
Together with the lower bound (\ref{lowerboundBin}), these two facts indeed lead to  
$$
n^{\gamma'}  \mathbb P\left(\frac{D_n-D_{n}(\lfloor nc \rfloor)}{a_n}>\gamma \frac{\ln(n)}{\ln(\ln(n))}\right) \underset{n \rightarrow \infty}\longrightarrow \infty
$$
since $ \gamma<(1-\varepsilon)^2 d^{\alpha+\varepsilon}\gamma'$.
We finish with a quick proof of (\ref{minoBinom}).
More generally, let $a,b,x>0$. Then
$$
\mathbb P\left(\mathrm{Bin}\left(\lfloor an \rfloor, \frac{b}{n}\right) > \frac{x\ln(n)}{\ln(\ln(n))} \right) \geq \binom{\lfloor an \rfloor}{\left \lfloor  \frac{x\ln(n)}{\ln(\ln(n))} \right\rfloor+1}\left(\frac{b}{n} \right)^{\left \lfloor  \frac{x\ln(n)}{\ln(\ln(n))} \right\rfloor+1}\left(1-\frac{b}{n}\right)^{\left \lfloor  \frac{x\ln(n)}{\ln(\ln(n))} \right\rfloor+1}.
$$
Using Stirling's formula, the binomial term rewrites
$$
\binom{\lfloor an \rfloor}{\left \lfloor  \frac{x\ln(n)}{\ln(\ln(n))} \right\rfloor+1}=\exp\left(\left( \left \lfloor  \frac{x\ln(n)}{\ln(\ln(n))} \right\rfloor+1 \right)\left(\ln(an)-\ln\left(\frac{x\ln(n)}{\ln(\ln(n))}\right)+1+\circ(1) \right) \right).
$$
Hence,
\begin{eqnarray*}
&&\mathbb P\left(\mathrm{Bin}\left(\lfloor an \rfloor, \frac{b}{n}\right) > \frac{x\ln(n)}{\ln(\ln(n))} \right) \\
&\geq& \exp\left(\left( \left \lfloor  \frac{x\ln(n)}{\ln(\ln(n))} \right\rfloor+1 \right)\left(\ln(an)-\ln\left(\frac{x\ln(n)}{\ln(\ln(n))}\right)+1+\ln\left(\frac{b}{n}\right)+\circ(1) \right) \right)\\
&=& \exp \big(-x \ln(n) (1+\circ(1))\big).
\end{eqnarray*}
$\hfill \square$

\subsection{Functional convergence}
\label{Secfunc}

In this section we prove Proposition \ref{propfunctional}. To lighten notation, we let for all $n \in \mathbb N$
$$
\xi_n(t):=\frac{D_n(\lfloor nt \rfloor)}{a_n}=\frac{\sum_{i=1}^{\lfloor nt \rfloor} a_i V_i \mathbbm 1_{\{U_i \leq a_i/A_i\}}}{a_n}, \quad  \quad 0 \leq t \leq 1,
$$
where $U_i,V_i,1 \leq i \leq n$ are i.i.d. uniform on $(0,1)$
(recall the construction (\ref{D_n(k)})). 
Our goal is to prove that the process $(\xi_n)$ converges to the process $\xi_{(\alpha)}$ defined by (\ref{processXi}) for the Skorokhod topology on $D([0,1],\mathbb R_+)$. We start by proving the finite--dimensional convergence, relying on manipulations done in Section \ref{SecOneDim}. Then we use Aldous' tightness criterion to conclude that the convergence holds with respect to the topology of Skorokhod. 

\medskip

\textbf{Finite--dimensional convergence.} The processes $\xi_n, n \geq 1$ and $\xi_{(\alpha)}$ all have independent increments, by construction. It remains to prove that 
$$
\xi_n(t)-\xi_n(s) \underset{n \rightarrow \infty}{\overset{(\mathrm d)}\longrightarrow} \xi_{(\alpha)}(t)-\xi_{(\alpha)}(s)
$$
for all $0 \leq s \leq t \leq 1$. From the proof of Lemma \ref{lemLaplace} (i), we immediately get that for all $\lambda \geq 0$
\begin{eqnarray*}
\mathbb E\left[ \exp\left(\lambda \big(\xi_n(t)-\xi_n(s)\big)\right)\right]&=&\prod_{i=\lfloor ns \rfloor+1}^{\lfloor nt \rfloor} \mathbb E \left[ \exp\left(\lambda \frac{a_i}{a_n} V \mathbbm 1_{\left\{U \leq \frac{a_i}{A_i} \right\}} \right) \right] \\
&\underset{n\rightarrow \infty}\rightarrow &  \exp\left((\alpha+1) \int_s^{t} \frac{1}{x}\left(\left(\frac{\exp(\lambda x^{\alpha})-1}{\lambda x^{\alpha}} \right)-1\right) \mathrm dx\right).
\end{eqnarray*}
On the other hand, Campbell's theorem applied to the Poisson point process $(t_i,v_i)$ on $(0,\infty)^2$   with intensity $(\alpha+1)t^{-\alpha-1} \mathbbm 1_{\{v \leq t^{\alpha}\}} \mathrm dt \mathrm dv $ implies that for all $\lambda \geq 0$
$$
\mathbb E\left[\exp\left( \lambda(\xi_{(\alpha)} (t)-\xi_{\alpha} (s)\right) \right]=\exp\left((\alpha+1) \int_{s}^{t} \int_0^{x^{\alpha}}\left(\exp(\lambda v)-1\right) \mathrm dv \frac{\mathrm dx}{x^{1+\alpha}}\right)
$$
which indeed coincides with the above limit of $\mathbb E\left[ \exp\left(\lambda \big(\xi_n(t)-\xi_n(s)\big)\right)\right]$. 

\medskip

\textbf{Tightness.} We use Aldous' tightness criterion (\cite[Theorem 16.10]{Bill99}) that ensures that $(\xi_n)$ is tight with respect to the  Skorokhod topology on $D([0,1],\mathbb R_+)$  if:  
\begin{enumerate}
\item[$\bullet$] $\lim_{c \rightarrow \infty} \limsup_{n \rightarrow \infty}\mathbb P(\sup_{t \in [0,1]} \xi_n(t) >c)=0$
\item[$\bullet$] and for all $\varepsilon>0$
\begin{equation}
\label{Aldous}
\lim_{\delta \rightarrow 0}\limsup _{n \rightarrow \infty}\sup_{\tau \in \mathsf S_n} \sup_{0 \leq \theta \leq \delta}\mathbb P\big(\left|\xi_n\left((\tau+\theta)\wedge 1\right)-\xi_n\left(\tau \right) \right|>\varepsilon \big) = 0
\end{equation}
where $\mathsf S_n$ is the set of stopping times with respect to the filtration generated by the process $\xi_n$. 
\end{enumerate}
The first point is obvious, since the processes $\xi_n$ are non--decreasing and we already know that $\xi_n(1)$ converges in distribution. For the second point, 
note that if $\tau \in \mathsf S_n$, then $\lfloor n \tau\rfloor$ is a stopping time with respect to the filtration generated by the process $(D_n(k),0 \leq k \leq n)$. Hence
\begin{eqnarray*}
\sup_{0 \leq \theta \leq \delta} \mathbb P\big(\left|\xi_n\left((\tau+\theta)\wedge 1\right)-\xi_n\left(\tau \right) \right|>\varepsilon\big) &=& 
\mathbb P \big( \xi_n((\tau+\delta)\wedge 1)-\xi_n(\tau)>\varepsilon \big) \\
&\leq& \sum_{k=0}^n \mathbb P(\lfloor n\tau \rfloor=k) \mathbb P\left(\sum_{i=k+1}^{k+1+\lfloor n\delta \rfloor}\frac{a_i}{a_n}V_i \mathbbm 1_{\{U_i \leq \frac{a_i}{A_i}\}}>\varepsilon  \right).
\end{eqnarray*}
We may assume that $\varepsilon<\alpha$. Then, using $\mathbb P(X>\varepsilon) \leq \varepsilon^{-1} \mathbb E[X]$ for any non--negative r.v. $X$, we get
\begin{eqnarray*}
\mathbb P\left(\sum_{i=k+1}^{k+1+\lfloor n\delta \rfloor)}\frac{a_i}{a_n}V^i \mathbbm 1_{\{U_i \leq \frac{a_i}{A_i}\}} >\varepsilon\right) 
&\leq& \frac{1}{2\varepsilon} \sum_{i=k+1}^{k+1+\lfloor n\delta \rfloor} \frac{a_i^2}{a_nA_i} \\
&\underset{\text{by } (\ref{regularvar1}),(\ref{regularvar2}), \text{ for $n \geq n_{\varepsilon}$ and all $k \leq n$}}\leq & \frac{C_{\alpha,\varepsilon}}{n^{\alpha-\varepsilon}}\sum_{i=k+1}^{k+1+ \lfloor n\delta \rfloor} i^{\alpha-\varepsilon-1} \\
&\underset{\text{for $n \geq n_{\varepsilon}$ and all $k\leq n$}}\leq & C_{\alpha,\varepsilon} \max (\delta^{\alpha-\varepsilon}, \delta)
\end{eqnarray*}
where $C_{\alpha,\varepsilon}$ depends only on $\alpha,\varepsilon$.
To get the last line we have used that either $\alpha-\varepsilon-1 \geq 0$ and then (since $k+1 \leq 2n$)
$$
\frac{1}{n^{\alpha-\varepsilon}}\sum_{i=k+1}^{k+1+ \lfloor n\delta \rfloor} i^{\alpha-\varepsilon-1} \leq \frac{((2+\delta)n)^{\alpha-\varepsilon-1} n\delta}{n^{\alpha-\varepsilon}}=(2+\delta)^{\alpha-\varepsilon-1} \delta.
$$
Or $\alpha-\varepsilon-1 < 0$ and then
$$
\frac{1}{n^{\alpha-\varepsilon}}\sum_{i=k+1}^{k+1+ \lfloor n\delta \rfloor} i^{\alpha-\varepsilon-1} \leq \frac{\min\big((k+1+n\delta)^{\alpha-\varepsilon},(k+1)^{\alpha-\varepsilon-1}n\delta \big)}{(\alpha-\varepsilon)n^{\alpha-\varepsilon}} \leq \frac{(2\delta)^{\alpha-\varepsilon}}{\alpha-\varepsilon}
$$
where the last inequality is obtained by considering the first term in the minimum when $k+1\leq n\delta$ and the second term when $k+1> n\delta$.

In conclusion, we have proved that for all $n$ large enough and all stopping times $\tau \in \mathsf S_n$,
$$
\sup_{0 \leq \theta \leq \delta} \mathbb P\big(\left|\xi_n\left((\tau+\theta)\wedge 1\right)-\xi_n\left(\tau \right) \right|>\varepsilon\big) \leq  C_{\alpha,\varepsilon} \max (\delta^{\alpha-\varepsilon}, \delta).
$$
which gives (\ref{Aldous}).

\section{Multiple marking}
\label{SecMM}

In order to prove Theorem \ref{thm:heightmax} (ii), we need the joint distribution of the paths from the root to two points marked independently, uniformly in the tree $T_n$. This is studied in Section \ref{sec:Marking2}.  Then in Section \ref{sec:Markingk}, we turn to $\ell$ marked points and the proof of Proposition \ref{prop:Prokho}. 

\subsection{Marking two points}
\label{sec:Marking2}

The result of this section are available for any sequence $(a_n)$ of positive terms. 

Given $T_n$, let $X_n^{(1)},X_n^{(2)}$ denote two points taken independently and uniformly in $T_n$, and  $D_n^{(1)},D_n^{(2)}$ their respective distances to the root. For all $1 \leq k \leq n$, let also $D^{(1)}_n(k)$ (resp. $D^{(2)}_n(k)$) denote the distance to the root of the projection of $X_n^{(1)}$ (resp.  $X_n^{(2)}$) onto $T_k \subset T_n$.
Our goal is to describe the joint distribution of the paths $\big(\big(D_n^{(1)}(k),D_n^{(2)}(k)\big),1\leq k \leq n\big)$ -- we recall that the marginals are given by (\ref{D_n(k)}). In that aim, we introduce a sequence $\big(B^{(i,1)},B^{(i,2)}\big),  i \geq 1$ of  independent pairs of random variables defined by:
\begin{equation}
\label{defpair}
\vspace{0.1cm} \left\{\begin{array}{ccc}\mathbb P\left((B^{(i,1)},B^{(i,2)})=(1,1)\right) & =  & 0 \\ \vspace{0.1cm}\mathbb P\left((B^{(i,1)},B^{(i,2)})=(1,0)\right) & = & \frac{a_i}{A_{i}+a_i} \\ \vspace{0.1cm}\mathbb P\left((B^{(i,1)},B^{(i,2)})=(0,1)\right) & = & \frac{a_i}{A_{i}+a_i} \\ \mathbb P\left((B^{(i,1)},B^{(i,2)})=(0,0)\right) & = & \frac{A_{i-1}}{A_{i}+a_i}.\end{array}\right.
\end{equation} 
Note the two following facts (which will be useful later on):

\noindent $\bullet$ $B^{(i,1)}$ (resp. $B^{(i,2)}$) is stochastically smaller than a Bernoulli r.v. with success parameter $a_i/A_i$

\noindent $\bullet$ the distribution of $B^{(i,1)}$  given that $B^{(i,2)}=0$ (resp. $B^{(i,2)}$  given that $B^{(i,1)}=0$) is a Bernoulli r.v. with success parameter $a_i/A_i$.

\bigskip

\begin{lem}
\label{Twomarked}
Let $U_i,V_i,V_i^{(1)},V_i^{(2)},i\geq 1$ be independent r.v. uniformly distributed on $(0,1)$, all independent of a sequence $((B^{(i,1)},B^{(i,2)}),i \geq 1)$ of independent pairs of Bernoulli r.v. distributed as (\ref{defpair}). Then for all $n \geq 1$ and all  bounded continuous functions $f:\mathbb R^{2 \times n}  \rightarrow \mathbb R$,
\begin{eqnarray}
\label{sumkappa}
&&\mathbb E\left[ f\left(\big(D_n^{(1)}(k), D_n^{(2)}(k)\big),1 \leq k \leq n \right) \right] \\
\nonumber
&=&\sum_{\kappa=1}^n \left(\frac{a_{\kappa}}{A_{\kappa}}\right)^2 \left(\prod_{i={\kappa+1}}^n \bigg(1-\left(\frac{a_i}{A_i}\right)^2 \bigg)\right) \times \mathbb E\left[ f \left(\big(\Delta_{n,\kappa}^{(1)}(k), \Delta_{n,\kappa}^{(2)}(k)\big), 1 \leq k \leq n\right) \right] \end{eqnarray}
where for $j=1,2$,
\begin{eqnarray}
\label{aftersplitting}
\Delta_{n,\kappa}^{(j)}(k)=\sum_{i=1}^{(\kappa-1)\wedge k}a_iV_i \mathbbm 1_{\left\{U_i \leq \frac{a_i}{A_i}\right\}}+a_{\kappa}V_{\kappa}^{(j)} \mathbbm 1_{\{k \geq \kappa\}}+\sum_{i=\kappa+1}^{k} a_i V_i^{(j)}B^{(i,j)}.
\end{eqnarray}
\end{lem}

This lemma implies in particular that the distribution of the splitting index $S_n(2)$ of the two paths linking respectively $X_n^{(1)}$ and $X_n^{(2)}$ to the root, i.e.
$$
S_n(2):=\inf \left\{1\leq k \leq n: p_k(X_n^{(1)})\neq p_k(X_n^{(2)}) \right\},
$$ 
where $p_k(X_n^{(i)}), i=1,2$ denotes the projection of $X_n^{(i)}$ onto $T_k$, 
is given by
\begin{equation}
\label{cvSn2}
\mathbb P\left(S_n(2)=\kappa \right)=\left(\frac{a_{\kappa}}{A_{\kappa}}\right)^2\prod_{i={\kappa+1}}^n \left(1-\left(\frac{a_i}{A_i}\right)^2 \right), \quad 1\leq \kappa \leq n
\end{equation}
(which is indeed a probability distribution!). Moreover, given $S_n(2)=\kappa$, the dependence of the two paths above the index $\kappa+1$ is only driven by pairs of random variables $\left(B^{(i,1)},B^{(i,2)}\right),i \geq \kappa+1$, as described in (\ref{aftersplitting}).

\begin{proof}
We proceed by induction on $n \geq 1$. For $n=1$, the formula of the lemma reduces to $$\mathbb E\left[ f\big(D_1^{(1)},D_1^{(2)}\big) \right]=\mathbb E \left[ f\big(a_1 V_1^{(1)},a_1 V_1^{(2)}\big)\right]$$ which is obviously true since the two marked points are independently and uniformly distributed on a segment of length $a_1$. Consider now an integer $n\geq 2$ and assume that the formula of the lemma holds for $n-1$.  When marking $X_n^{(1)},X_n^{(2)}$, four disjoint situations may arise:
\begin{enumerate}
\item[$\bullet$] with probability $(a_n/A_n)^2$, the two marked points are on the branch $\mathsf b_n$. Conditionally on this event,  $D^{(1)}_n(k)=D^{(2)}_n(k), 1 \leq k \leq n-1$ which corresponds to the path to the root of a point uniformly distributed on $T_{n-1}$, which is distributed as 
$$
\sum_{i=1}^{k}a_iV_i \mathbbm 1_{\left\{U_i \leq \frac{a_i}{A_i}\right\}}, \quad 1 \leq k \leq n-1.
$$
Moreover $D^{(1)}_n(n)-D^{(1)}_n(n-1)$ and $D^{(2)}_n(n)-D^{(1)}_n(n-1)$ are independent, independent of the path $(D^{(1)}_n(k), k \leq n-1)$, and uniformly distributed on  $\mathsf b_n$, which has length $a_n$.
All this leads to the term $\kappa=n$ in the sum (\ref{sumkappa}).
\item[$\bullet$] with probability $A_{n-1}a_n/A^2_n$, $X_n^{(1)} \in T_{n-1}$ and $X_n^{(2)} \in \mathsf b_n$. Conditionally on this event, $D_n^{(1)}(k),1 \leq k\leq n-1$ and $D_n^{(2)}(k),1 \leq k\leq n-1$ correspond to the respective paths to the root of two points marked independently, uniformly in $T_{n-1}$. Their joint distribution is therefore given by the induction hypothesis. Moreover $D_n^{(1)}(n)=D_n^{(1)}(n-1)$ and $D_n^{(2)}(n)-D_n^{2}(n-1)$ is independent of the paths $(D_n^{(1)}(k),D_n^{(2)}(k)),1 \leq k\leq n-1$ and is uniformly distributed on $\mathsf b_n$. To sum up, setting $\Delta_{n-1,\kappa}^{(1)}(n):=\Delta_{n-1,\kappa}^{(1)}(n-1)$ and setting for $\kappa \leq n-1$ $\Delta_{n-1,\kappa}^{(2)}(n):=\Delta_{n-1,\kappa}^{(2)}(n-1)+a_nV_n^{(2)}$, we have:
\begin{eqnarray*}
&&\mathbb E\left[ f\left(\big(D_n^{(1)}(k), D_n^{(2)}(k)\big),1 \leq k \leq n \right) \mathbbm 1_{\left\{X_n^{(1)} \in T_{n-1},X_n^{(2)} \in \mathsf b_n\right\}}\right] \\
&=& \frac{A_{n-1}a_n}{A^2_n} \times  \sum_{\kappa=1}^{n-1} \left(\frac{a_{\kappa}}{A_{\kappa}}\right)^2 \left(\prod_{i={\kappa+1}}^{n-1} \bigg(1-\left(\frac{a_i}{A_i}\right)^2 \bigg)\right)  \\
&&  \hspace{2.5cm}  \times \mathbb E\left[ f\left(\big(\Delta_{n-1,\kappa}^{(1)}(k), \Delta_{n-1,\kappa}^{(2)}(k)\big), 1 \leq k \leq n \right) \right] \\
&=& \sum_{\kappa=1}^{n-1} \left(\frac{a_{\kappa}}{A_{\kappa}}\right)^2 \left(\prod_{i={\kappa+1}}^n \bigg(1-\left(\frac{a_i}{A_i}\right)^2 \bigg)\right) \\
&& \hspace{0.7cm} \times \mathbb E\left[ f \left(\big(\Delta_{n,\kappa}^{(1)}(k), \Delta_{n,\kappa}^{(2)}(k)\big),1 \leq k \leq n \right) \mathbbm 1_{\left\{B^{(n,1)}=0,B^{(n,2)}=1 \right\}} \right],
\end{eqnarray*}
where we have used for the second equality that 
$$
 \frac{A_{n-1}a_n}{A^2_n} =\left(1-\left(\frac{a_n}{A_n} \right)^2 \right) \times \mathbb P\big(B^{(n,1)}=0,B^{(n,2)}=1\big).
$$
\item[$\bullet$] with probability $A_{n-1}a_n/A^2_n$, $X_n^{(2)} \in T_{n-1}$ and $X_n^{(1)} \in \mathsf b_n$, which is  symmetric to the previous case. 
\item[$\bullet$] with probability $(A_{n-1}/A_n)^2$ the two marked points are in $T_{n-1}$. Conditionally on this event,  $D^{(1)}_n(n)=D^{(1)}_n(n-1)$,  $D^{(2)}_n(n)=D^{(2)}_n(n-1)$  and $D_n^{(1)}(k),1 \leq k\leq n-1$ and $D_n^{(2)}(k),1 \leq k\leq n-1$ correspond to the paths to the root  of  two points marked independently, uniformly in $T_{n-1}$. Their joint distribution is therefore given by the induction hypothesis, and setting for $\kappa \leq n-1$ $\Delta^{(1)}_{n-1,\kappa}(n):=\Delta^{(1)}_{n-1,\kappa}(n-1)$ and $\Delta^{(2)}_{n-1,\kappa}(n):=\Delta^{(2)}_{n-1,\kappa}(n-1)$, we have:
\begin{eqnarray*}
&&\mathbb E\left[ f\left(\big(D_n^{(1)}(k), D_n^{(2)}(k)\big),1 \leq k \leq n \right) \mathbbm 1_{\left\{X_n^{(1)} \in T_{n-1},X_n^{(2)} \in T_{n-1} \right\}}\right] \\
&=& \frac{A_{n-1}^2}{A^2_n} \times  \sum_{\kappa=1}^{n-1} \left(\frac{a_{\kappa}}{A_{\kappa}}\right)^2 \left(\prod_{i={\kappa+1}}^{n-1} \bigg(1-\left(\frac{a_i}{A_i}\right)^2 \bigg)\right)  \\
&& \hspace{2.1cm} \times \mathbb E\left[ f\left(\big(\Delta_{n-1,\kappa}^{(1)}(k), \Delta_{n-1,\kappa}^{(2)}(\kappa)\big), 1 \leq k \leq n\right) \right] \\
&=& \sum_{\kappa=1}^{n-1} \left(\frac{a_{\kappa}}{A_{\kappa}}\right)^2 \left(\prod_{i={\kappa+1}}^n \bigg(1-\left(\frac{a_i}{A_i}\right)^2 \bigg)\right) \\
&& \hspace{0.7cm} \times \mathbb E\left[ f \left(\big(\Delta_{n,\kappa}^{(1)}(k), \Delta_{n,\kappa}^{(2)}(k)\big) ,1 \leq k \leq n\right)\mathbbm 1_{\left\{B^{(n,1)}=0,B^{(n,2)}=0 \right\}} \right],
\end{eqnarray*}
where we have used for the second equality that 
$$
 \frac{A_{n-1}^2}{A^2_n} =\left(1-\left(\frac{a_n}{A_n} \right)^2 \right) \times \mathbb P\big(B^{(n,1)}=0,B^{(n,2)}=0\big).
$$
\end{enumerate}
Gathering these four situations finally leads to the formula of the lemma for $n$.
\end{proof}

\subsection{Marking $\ell$ points and behavior of $T_n(\ell)$}
\label{sec:Markingk}

The goal of this section is to prove Proposition \ref{prop:Prokho}. 
We start with a few notation. For each $n$, given $T_n$, let $X_n^{(1)},\ldots,X_n^{(\ell)}$ be $\ell$ points picked independently and uniformly in $T_n$. Let then $D_n^{(1)},\ldots,D_n^{(\ell)}$ be their respective distances to the root, and for all $1 \leq k \leq n$,  $D^{(1)}_n(k),\ldots, D^{(\ell)}_n(k)$ be the respective distances to the root of the projections of $X_n^{(1)},\ldots,X_n^{(\ell)}$ onto $T_{k} \subset T_n$. 

In the tree $T_n(\ell)$, the subtree of $T_n$ spanned from the root and $X_n^{(1)},\ldots,X_n^{(\ell)}$, we let, using the notation of the introduction, 
$$\mathsf B_n(\ell):=\mathsf B_n^{(i_0,j_0)} \quad \text{if} \quad d(\mathsf B_n^{(i_0,j_0)},\mathsf{root})=\max_{1 \leq i \neq j\leq \ell}(d(\mathsf B_n^{(i,j)},\mathsf{root}))$$
be the point amongst the $\mathsf B_n^{(i,j)}, 1 \leq i\neq j\leq \ell$ the farthest from the root (note that it is well--defined a.s.). We may and will also see $\mathsf B_n(\ell)$ as a point of $T_n$.

We will need the following random variables. For all $i \geq 1$, let $\left(B^{(i,1)},\ldots,B^{(i,\ell)}\right)$ be an exchangeable $\ell$--uplet with distribution 
\begin{equation}
\label{defnuplet}
\left\{\begin{array}{lll} \vspace{0.15cm}\mathbb P\left((B^{(i,1)},\ldots,B^{(i,\ell)})=(u_1,\ldots,u_{\ell})\right) & =  & 0 \quad \text{ for all } (u_i)_{1 \leq i \leq \ell}  \in \{0,1\}^\ell \text{ with at least two 1} \\  \vspace{0.15cm} \mathbb P\left((B^{(i,1)},\ldots,B^{(i,\ell)})=(1,0,\ldots,0)\right) & = & \frac{a_i}{A_{i-1}+ \ell a_i} \\ \mathbb P\left((B^{(i,1)},\ldots,B^{(i,\ell)})=(0,0\ldots,0)\right) & = & \frac{A_{i-1}}{A_{i-1}+\ell a_i}. 
\end{array}\right.
\end{equation} 

In order to study the asymptotic behavior of $(T_n(\ell))$, we set up the following lemma, which is similar to Lemma \ref{Twomarked}, although less explicit.

\begin{lem}
\label{lemcond}
For all $k \in \mathbb N$ and all $n \in \mathbb N$, $n>k$, the distribution of 
$$
\big(D_n^{(1)}-D_{k+1}^{(1)},\ldots,D_n^{(\ell)}-D_{k+1}^{(\ell)} \big) \text{ given that } \mathsf B_n(\ell) \in T_{k}
$$
is the same as that of 
$$
\left(\sum_{i=k+2}^n a_i V_i^{(1)}B^{(i,1)}, \ldots, \sum_{i=k+2}^n a_i V_i^{(\ell)}B^{(i,\ell)} \right),
$$
where the random variables $V_i^{(j)},i\geq 1,1 \leq j \leq \ell$ are i.i.d. uniform on $(0,1)$, the $\ell$--uplets $\left(B^{(i,1)},\ldots,B^{(i,\ell)}\right)$ are distributed via \emph{(}\ref{defnuplet}\emph{)}, $\forall i \geq 1$, independently of each other and independently of $(V_i^{(j)},i\geq 1,1 \leq j \leq \ell)$.
\end{lem}

\begin{proof} The proof is similar to that of Lemma \ref{Twomarked} and holds by induction on $n>k$. We sketch it briefly. For $n=k+1$ the statement is obvious since both $\ell-$uplets are then equal to $(0,\ldots,0)$. Assume now that the statement holds for some $n>k$. Then observe what happens for $n+1$: given that $\mathsf B_{n+1}(\ell) \in T_{k}$, two situations may occur:
\begin{enumerate}
\item[$\bullet$] either none of the marked points belongs to the segment $\mathsf b_{n+1}$. This occurs with a probability proportional to $(A_n)^{\ell}$ and then   
$$
\big(D_{n+1}^{(1)}-D_{k+1}^{(1)},\ldots,D_{n+1}^{(\ell)}-D_{k+1}^{(\ell)} \big) \text{ given that } \mathsf B_{n+1}(\ell) \in T_{k}
$$
is distributed as 
$$
\big(D_n^{(1)}-D_{k+1}^{(1)},\ldots,D_n^{(\ell)}-D_{k+1}^{(\ell)} \big) \text{ given that } \mathsf B_n(\ell) \in T_{k}.
$$
\item[$\bullet$]  or a unique marked point belongs to the segment $\mathsf b_{n+1}$. The probability that $X_{n+1}^{(1)}$ belongs to $\mathsf b_{n+1}$ (and not the other $\ell-1$ marked points) is proportional to $a_{n+1}(A_n)^{\ell-1}$ and in that case,
$$
\big(D_{n+1}^{(1)}-D_{k+1}^{(1)},\ldots,D_{n+1}^{(\ell)}-D_{k+1}^{(\ell)} \big) \text{ given that } \mathsf B_{n+1}(\ell) \in T_{k}
$$
is distributed as 
$$
\big(D_n^{(1)}+a_{n+1}V-D_{k+1}^{(1)},\ldots,D_n^{(\ell)}-D_{k+1}^{(\ell)} \big) \text{ given that } \mathsf B_n(\ell) \in T_{k},
$$
where $V$ is uniform on $(0,1)$ and independent of $D_n^{(i)}-D_{k+1}^{(i)}, 1 \leq i \leq \ell,  \mathsf B_n(\ell)$.
\end{enumerate}
This leads to the statement for $n+1$.
\end{proof}

\bigskip

\textit{Proof of Proposition \ref{prop:Prokho}.} Throughout this proof it is assumed that $(a_n)$ is regularly varying with index $\alpha>0$ (the proof is identical under the assumptions (ii) of Proposition \ref{prop:Prokho}).  With the notation of this section, our goal is to prove that 
$$
\left(\frac{D_n^{(i)}}{a_n}, 1 \leq i \leq \ell, \frac{d(\mathsf B_n(\ell),\mathsf{root})}{a_n}\right)\underset{n \rightarrow \infty}{\overset{\mathrm{(d)}}\longrightarrow} \left(\left(\xi^{(i)}_{(\alpha)},1 \leq i \leq \ell\right),0 \right) 
$$ 
where $\xi^{(1)}_{(\alpha)},\ldots,\xi^{(\ell)}_{(\alpha)}$ are i.i.d. with distribution (\ref{TransfoLapl}).
We first claim that
\begin{equation*}
\label{cvBn}
\frac{d(\mathsf B_n(\ell),\mathsf{root})}{a_n} \underset{n \rightarrow \infty}{\overset {\mathbb P} \longrightarrow} 0,
\end{equation*}
since $a_n \rightarrow \infty$ and $d(\mathsf B_n(\ell),\mathsf{root}) \leq \sum_{1 \leq i \neq j \leq \ell} d(\mathsf B_n^{(i,j)},\mathsf{root})$, which is stochastically bounded since the splitting index $S_n(2)$ of the paths of two marked points converges in distribution, by (\ref{cvSn2}). By Slutsky's Theorem, it remains to prove that   
$$
\left(\frac{D_n^{(i)}}{a_n}, 1 \leq i \leq \ell \right)\underset{n \rightarrow \infty}{\overset{\mathrm{(d)}}\longrightarrow} \left(\xi^{(i)}_{(\alpha)},1 \leq i \leq \ell\right). 
$$ 
We start by observing that for all $k\geq 1$, 
$$
\left(\frac{D_n^{(i)}-D_{k+1}^{(i)}}{a_n}, 1 \leq i \leq \ell \right) \quad \text{given that } \mathsf B_n(\ell) \in T_{k}\quad \underset{n \rightarrow \infty}{\overset{\mathrm{(d)}}\longrightarrow} \left(\xi^{(i)}_{(\alpha)},1 \leq i \leq \ell\right),  
$$
which obviously leads to (since $a_n \rightarrow \infty$)
$$
\left(\frac{D_n^{(i)}}{a_n}, 1 \leq i \leq \ell \right) \quad \text{given that } \mathsf B_n(\ell) \in T_{k}\quad \underset{n \rightarrow \infty}{\overset{\mathrm{(d)}}\longrightarrow} \left(\xi^{(i)}_{(\alpha)},1 \leq i \leq \ell\right).  
$$
The above observation relies on the following consequence of Lemma \ref{lemcond}: for all $(\lambda_i)_{1 \leq i \leq \ell} \in \mathbb R^{\ell}$ and all $n>k$,
$$
\ln \left( \mathbb E \left[ \exp\left(\sum_{i=1}^{\ell} \lambda_i \frac{D_n^{(i)}-D_{k+1}^{(i)}}{a_n}\right)\right] |  \mathsf B_n(\ell) \in T_{k}\right)=\sum_{j=k+2}^n  \ln \left(1+  \frac{a_j}{A_{j-1}+\ell a_j}\sum_{i=1}^{\ell} \frac{\exp(\lambda_i \frac{a_j}{a_n})-1}{\lambda_i \frac{a_j}{a_n}}-1 \right)
$$
(with the usual convention $x^{-1}(\exp(x)-1)=1$ when $x=0$).
A slight modification of the proof of Lemma \ref{lemLaplace} implies that this logarithm converges to  $\sum_{i=1}^{\ell} \phi_{(\alpha)}(\lambda_i)$, which then leads to the expected convergences in distribution. The end of the proof is then easy. Let $\mathbf V_n=(a_n^{-1} D_n^{(i)}, 1 \leq i \leq \ell)$ and $f:\mathbb R^{\ell} \rightarrow \mathbb R$ be a continuous, bounded function. Fix $\varepsilon >0$. There exists $k_{\varepsilon} \in \mathbb N$ such that $\mathbb P(\mathsf B_n(\ell) \notin T_{k_{\varepsilon}}) \leq \varepsilon$ for all $n$, since, as already mentioned, the splitting index of the paths of two marked points converges in distribution, by (\ref{cvSn2}). Then, writing
$$
\mathbb E\left[f(\mathbf V_n) \right] =\mathbb E\left[f(\mathbf V_n) | B_n(\ell) \in T_{k_{\varepsilon}} \right] \mathbb P(\mathsf B_n(\ell) \in T_{k_{\varepsilon}}) +\mathbb E\left[f(\mathbf V_n) \mathbbm 1_{\{\mathsf B_n(\ell) \notin T_{k_{\varepsilon}}\}}\right] 
$$
we get that
\begin{eqnarray*}
\mathbb E\left[ f\big(\xi^{(i)}_{(\alpha)},1 \leq i \leq \ell\big)\right](1-\varepsilon) - \sup_{x \in \mathbb R^{\ell}} |f(x)| \varepsilon &\leq& 
\liminf_{n \rightarrow \infty} \mathbb E\left[f(\mathbf V_n) \right] \\ &\leq& \limsup_{n \rightarrow \infty} \mathbb E\left[f(\mathbf V_n) \right] \leq \mathbb E\left[ f\big(\xi^{(i)}_{(\alpha)},1 \leq i \leq \ell\big)\right] + \sup_{x \in \mathbb R^{\ell}} |f(x)| \varepsilon.
\end{eqnarray*}
Letting $\varepsilon \rightarrow 0$ gives the result. 
$\hfill \square$

\section{Height of $T_n$ when $\alpha>0$} 
\label{SecHeight}

Throughout this section we assume that $(a_n)$ is regularly varying with index $\alpha>0$. Our goal is to prove that
$$
\frac{H_n \cdot \ln(\ln(n))}{a_n \ln(n)} \overset{\mathrm{a.s.}}{\underset{n \rightarrow \infty}\longrightarrow} 1
$$
(Theorem \ref{thm:heightmax} (ii)).
We split the proof into two parts, starting with the fact that
\begin{eqnarray}
\label{upperbound}
\limsup_{n \rightarrow \infty} \frac{H_n \cdot \ln(\ln(n))}{a_n \ln(n)} \leq  1 \quad \text{a.s.},
\end{eqnarray}
which is an easy consequence of Borel--Cantelli's lemma and Lemma \ref{lem:equiv} (i). We will  then show that 
\begin{eqnarray}
\label{lowerbound}
\liminf_{n \rightarrow \infty} \frac{H_n \cdot \ln(\ln(n))}{a_n \ln(n)} \geq  1 \quad \text{a.s.},
\end{eqnarray}
using the second moment method and, again, Borel--Cantelli's lemma. To carry this out, we will use Lemma \ref{Twomarked} on the two marked points, as well as the estimates of Lemma \ref{lem:equiv} (ii).

\subsection{Proof of the limsup (\ref{upperbound})} 

In the infinite tree $\cup_{n \geq 1} T_n$, label the leaves by order of apparition: for each $i\geq 1$, the leaf $L_i$ is the one that belongs to the branch $\mathsf b_i$. Then consider for $i \geq 2$ the projection of $L_i$ onto $T_{i-1}$ and denote by $\overline D_{i-1}$ the distance of this projection to the root, which is distributed as $D_{i-1}$. Let $\overline D_0=0$ and note that
$$
H_n=\max_{1\leq i \leq n}\{d(L_i,\mathsf{root})\}=\max_{1 \leq i \leq n}\{\overline D_{i-1}+a_i\}.
$$
Now, let $c_1>c_2>1$. By Lemma \ref{lem:equiv} (i), 
$$
\sum_{i\geq 1} \mathbb P \left(\overline D_{i-1} \geq c_2 \frac{a_{i-1} \ln(i-1)}{\ln(\ln(i-1))} \right) <\infty.
$$
Hence by Borel--Cantelli's lemma, almost surely
$$
\overline D_{i-1} < c_2 \frac{a_{i-1} \ln(i-1)}{\ln(\ln(i-1))} 
$$
for all $i$ large enough. This leads, together with the fact that $(a_i)$ is regularly varying -- see in particular (\ref{regularvar1}) -- to the almost sure existence of a (random) $i_0$ such that
$$
\overline D_{i-1} +a_i< c_1 \frac{a_{n} \ln(n)}{\ln(\ln(n))} 
$$
for all $n\geq i \geq i_0$. Hence,
$$
\limsup_{n\rightarrow \infty} \frac{H_n \cdot \ln(\ln(n))}{a_n \ln(n)} \leq c_1 \quad \text{a.s.}
$$
This holds for all $c_1>1$, hence (\ref{upperbound}). 

\bigskip

\subsection{Proof of the liminf (\ref{lowerbound})} 

Let $X_n^{(i)},1 \leq i \leq n$ be $n$ points marked independently and uniformly  in $T_n$. Then let $D_n^{(i)},1 \leq i \leq n$ denote their respective distances to the root, and for all $k<n$, $D_n^{(i)}(k),1 \leq i \leq n$ denote the distances to the root of their respective projections onto $T_k$. Of course, $H_n \geq \max_{1 \leq i \leq n}D_n^{(i)}$ and it is sufficient to prove the liminf for this maximum of \emph{dependent} random variables. In that aim, we first settle the following lemma, using the second moment method.

\medskip

\begin{lem}
\label{lemmamin}
For all $ c \in (0,1)$ and all $\gamma<1$,
\begin{eqnarray}
\label{min} 
\mathbb P \left( \frac{\max_{1\leq i\leq n} \big(D_n^{(i)}-D_n^{(i)}(\lfloor nc \rfloor)\big)}{a_n} \leq \gamma \frac{\ln(n)}{\ln(\ln(n))}\right)\leq n^{\gamma-1+\circ(1)}.
\end{eqnarray}
\end{lem}
Since $H_n$ is larger than the maximum involved in this probability, this immediately implies that 
$$\mathbb P \left( \frac{H_n \cdot \ln(\ln(n))}{a_n \ln(n)} \leq \gamma \right) \underset{n \rightarrow \infty}{\longrightarrow} 0 \quad \text{for all } \gamma<1.$$
This is however not sufficient since we want an \emph{almost sure} bound for the liminf (\ref{lowerbound}). We will turn to this conclusion later on. We first prove the lemma. 

\noindent \textbf{Proof of Lemma \ref{lemmamin}.} We start with standard arguments, in order to use the second moment method. Fix $\gamma \in (0,1)$ and introduce
$$
A_n^{(i)}:=\left\{\frac{D_n^{(i)}-D_n^{(i)}(\lfloor nc \rfloor) }{a_n}> \gamma \frac{\ln(n)}{\ln(\ln(n))}\right\}, \quad 1 \leq i \leq n,
$$
and
$$
S_n:=\sum_{i=1}^n \mathbbm 1_{A_n^{(i)}}. 
$$
Since the sequence $(D_n^{(i)}-D_n^{(i)}(\lfloor nc \rfloor),1 \leq i \leq n)$ is exchangeable, we have: 
$$
\mathbb E \left[S_n\right]=n \mathbb P\big(A_n^{(1)}\big)$$ 
and
$$
\mathrm{Var}\left(S_n\right)=n  \mathbb P\big(A_n^{(1)}\big)+n(n-1) \mathbb P\big(A_n^{(1)}\cap A_n^{(2)}\big)-\big(n \mathbb P\big(A_n^{(1)}\big)\big)^2.
$$
Note that with this notation, (\ref{min}) rewrites
$
\mathbb P\left(S_n=0\right)\leq n^{\gamma-1+\circ(1)}.
$ 
To prove this upper bound, we use the second moment method:
\begin{eqnarray*}
\mathbb P\left(S_n=0\right)
&\leq&\frac{\mathrm{Var}\left(S_n\right)}{\left(\mathbb E[S_n]\right)^2} \\
&\leq & \frac{1}{n \mathbb P\big(A_n^{(1)}\big)}+\frac{ \mathbb P\big(A_n^{(1)}\cap A_n^{(2)}\big)}{\big(\mathbb P\big(A_n^{(1)}\big)\big)^2}-1.
\end{eqnarray*}
By Lemma \ref{lem:equiv} (ii), we know that $n\mathbb P\big(A_n^{(1)}\big)= n^{1-\gamma+\circ(1)}$. It remains to show that
$$
\frac{\mathbb P\big(A_n^{(1)}\cap A_n^{(2)}\big)}{\big(\mathbb P\big(A_n^{(1)}\big)\big)^2} \leq 1+n^{\gamma-1+\circ(1)}.
$$ 
In that aim, recall the notation and statement of Lemma \ref{Twomarked}:
\begin{eqnarray*}
 \mathbb P\big(A_n^{(1)} \cap A_n^{(2)}\big) 
 =  \sum_{\kappa=1}^{n} p_{\kappa} \mathbb P\left(\frac{\Delta_{n,\kappa}^{(j)}(n)- \Delta_{n,\kappa}^{(j)}(\lfloor nc \rfloor)}{a_n} >\gamma \frac{ \ln(n)}{\ln(\ln(n))}, j=1,2 \right) 
\end{eqnarray*}
where $p_{\kappa}:=\big(\frac{a_{\kappa}}{A_{\kappa}}\big)^2 \big(\prod_{i={\kappa+1}}^n \big(1-\big(\frac{a_i}{A_i}\big)^2 \big)\big)$ for $1 \leq \kappa \leq n$. We split this sum  into two parts:

(i) First, using the notation of Section \ref{sec:Marking2} and the remarks just before Lemma \ref{Twomarked}, we see that
\begin{eqnarray*}
&& \sum_{\kappa=1}^{\lfloor nc\rfloor} p_{\kappa} \mathbb P\left(\frac{\Delta_{n,\kappa}^{(j)}(n)- \Delta_{n,\kappa}^{(j)}(\lfloor nc \rfloor)}{a_n} >\gamma \frac{ \ln(n)}{\ln(\ln(n))},j=1,2\right)  \\
&=&  \sum_{\kappa=1}^{\lfloor nc\rfloor} p_{\kappa} \mathbb E\Bigg[\mathbbm 1_{\left\{\frac{\Delta_{n,\kappa}^{(1)}(n)- \Delta_{n,\kappa}^{(1)}(\lfloor nc \rfloor)}{a_n} >\gamma \frac{ \ln(n)}{\ln(\ln(n))}\right\}}\Bigg.  \\
&& \hspace{2.15cm} \Bigg.  \times \mathbb P\Bigg( \frac{\sum_{i=\lfloor nc\rfloor+1}^n a_i V_i^{(2)}B^{(i,2)}}{a_n} >\gamma \frac{\ln(n)}{\ln(\ln(n))}  \  \big| \ B^{(i,1)}, V^{(1)}_i, 1 \leq i \leq n\Bigg)\Bigg]  \\
& \leq & \mathbb P\big(A_n^{(1)} \big) \sum_{\kappa=1}^{\lfloor nc\rfloor} p_{\kappa} \mathbb P\left(\frac{\Delta_{n,\kappa}^{(1)}(n)- \Delta_{n,\kappa}^{(1)}(\lfloor nc \rfloor)}{a_n} >\gamma \frac{ \ln(n)}{\ln(\ln(n))} \right)\\
&\leq & \mathbb P\big(A_n^{(1)}\big)^2.
\end{eqnarray*}
The first inequality  is due to the fact that the sum $\sum_{i=\lfloor nc\rfloor+1}^n a_i V_i^{(2)}B^{(i,2)}$ given $B^{(i,1)},V_i^{(1)}$, $1 \leq i \leq n$  is stochastically smaller than  $D^{(1)}_n-D^{(1)}_n(\lfloor nc \rfloor)$ since the distribution of $B^{(i,2)}$ conditional on $B^{(i,1)}=0$ is a Bernoulli r.v. with success parameter $a_i/A_i$, and moreover $B^{(i,2)}=0$ a.s. when $B^{(i,1)}=1$.  The second inequality follows immediately from Lemma \ref{Twomarked}.

(ii) Second,
\begin{eqnarray*}
&& \sum_{\kappa=\lfloor nc\rfloor+1}^n p_{\kappa} \mathbb P\left(\frac{\Delta_{n,\kappa}^{(j)}(n)- \Delta_{n,\kappa}^{(j)}(\lfloor nc \rfloor)}{a_n} >\gamma \frac{\ln (n)}{\ln(\ln(n))}, j=1,2 \right) \\
&\leq&  \sum_{\kappa=\lfloor nc\rfloor+1}^n p_{\kappa} \mathbb P\left(\frac{\sum_{i=\kappa+1}^n a_iV_i^{(1)}B^{(i,1)}+a_{\kappa} V_{\kappa}^{(1)}+\sum_{i=\lfloor nc \rfloor+1}^{\kappa-1} a_i V_i \mathbbm 1_{\{U_i \leq a_i/A_i\}}}{a_n} >\gamma \frac{\ln (n)}{\ln(\ln(n))}\right) \\
 & \leq & n^{-\gamma+\circ(1)} \sum_{\kappa=\lfloor nc\rfloor+1}^n p_{\kappa} = n^{-\gamma-1+\circ(1)}.
 \end{eqnarray*}
Indeed, note that
$$
\sum_{i=\kappa+1}^n a_iV_i^{(1)}B^{(i,1)}+a_{\kappa} V_{\kappa}^{(1)}+\sum_{i=\lfloor nc\rfloor+1}^{\kappa-1} a_i V_i \mathbbm 1_{\{U_i \leq a_i/A_i\}}
$$
is stochastically dominated by
$
a_{\kappa}+D^{(1)}_n-D^{(1)}_n(\lfloor nc \rfloor)
$
since $B^{(i,1)}$ is dominated by a Bernoulli r.v. with success parameter $a_i/A_i$, for all $i$.
So by Lemma \ref{lem:equiv} (ii) and the fact that  $a_{\kappa}\leq 2 a_n$ uniformly in $\kappa \in \{\lfloor nc \rfloor,  \ldots, n \}$ for $n$ large enough (see (\ref{regularvar1})), we get that
$$
\mathbb P\left(\frac{\sum_{i=\kappa+1}^n a_iV_i^{(1)}B^{(i,1)}+a_{\kappa} V_{\kappa}^{(1)}+\sum_{i=\lfloor nc \rfloor+1}^{\kappa} a_i V_i \mathbbm 1_{\{U_i \leq a_i/A_i\}}}{a_n} >\gamma \frac{\ln (n)}{\ln(\ln(n))}\right) \leq n^{-\gamma +\circ(1)}
$$
with a $\circ(1)$ independent of $\kappa  \in \{\lfloor nc \rfloor, \ldots, n \}$.  Moreover, by (\ref{regularvar2}), 
$$\sum_{\kappa=\lfloor nc \rfloor+1}^n p_{\kappa} \leq \sum_{\kappa=\lfloor nc \rfloor+1}^n \Big(\frac{a_{\kappa}}{A_{\kappa}}\Big)^2=n^{-1+\circ(1)}.$$

Finally, gathering the two upper bounds established in (i) and (ii) and using again that $\mathbb P(A^{(1)}_n)=n^{-\gamma+\circ(1)}$, we have proved that
$$
\frac{\mathbb P\big(A_n^{(1)}\cap A_n^{(2)}\big)}{\big(\mathbb P\big(A_n^{(1)}\big)\big)^2} \leq 1+\frac{n^{-\gamma-1+\circ(1)}}{n^{-2\gamma+\circ(1)}}=1+n^{\gamma-1+\circ(1)}
$$
as wanted.
$\hfill \square$

\bigskip
It remains to deduce (\ref{lowerbound}) from Lemma \ref{lemmamin}. In that aim fix $\gamma \in (0,1)$. A first consequence of  Lemma \ref{lemmamin} is that
\begin{equation}
\label{rec1Hn}
\mathbb P\left(\frac{H_n}{a_n} \leq \gamma \frac{\ln(n)}{\ln(\ln(n))} \right) \leq n^{\gamma-1+\circ(1)}.
\end{equation}
Now let $c \in (0,1)$ and note that $$H_n \geq \max \Big(H_{\lfloor nc \rfloor}, \max_{1\leq i\leq n} \big(D_n^{(i)}-D_n^{(i)}(\lfloor nc \rfloor)\big)\Big)$$ with $H_{\lfloor nc \rfloor}$ and $\max_{1\leq i\leq n} \big(D_n^{(i)}-D_n^{(i)}(\lfloor nc \rfloor)\big)$ \emph{independent}. Hence,
\begin{eqnarray*}
\mathbb P\left(\frac{H_n}{a_n} \leq \gamma \frac{\ln(n)}{\ln(\ln(n))} \right) &\leq& \mathbb P\left(\frac{H_{\lfloor nc \rfloor}}{a_n} \leq \gamma \frac{\ln(n)}{\ln(\ln(n))} \right) \\
&& \hspace{0.0cm}\times \hspace{0.15cm}\mathbb P \left( \frac{\max_{1\leq i\leq n} \big(D_n^{(i)}-D_n^{(i)}(\lfloor nc \rfloor)\big)}{a_n} \leq \gamma \frac{\ln(n)}{\ln(\ln(n))}\right) \\
&\leq& n^{\gamma c^{-\alpha}-1+\circ(1)} \cdot n^{\gamma-1+\circ(1)}
\end{eqnarray*}
by (\ref{rec1Hn}) applied to $\lfloor nc \rfloor$ instead of $n$ (together with the regular variation assumption on $(a_n)$) and Lemma \ref{lemmamin}.
Next, fix an integer  $k$  such that $(1-\gamma)k>1$. Iterating the previous argument, we get that
$$
\mathbb P\left(\frac{H_n}{a_n} \leq \gamma \frac{\ln(n)}{\ln(\ln(n))} \right) \leq n^{\gamma \sum_{j=0}^{k-1}c^{-\alpha j}-k +\circ(1)}.
$$
We  now choose $c \in (0,1)$ sufficiently close to 1 so that $\gamma \sum_{j=0}^{k-1}c^{-\alpha j}-k<-1$ and conclude with Borel--Cantelli lemma that almost surely 
$$
\frac{H_n \cdot \ln(\ln(n))}{a_n \ln(n)} > \gamma \quad \text{for all $n$ large enough}.
$$
This holds for all $\gamma<1$. Hence (\ref{lowerbound}).

\section{The case $a_n=1$}
\label{Sec1}

The goal of this section is to prove Theorem \ref{thm:BroutinDevroye}. In that aim we start by associating to a sequence $(T_n)$ built recursively from a sequence $(a_n)$ of positive lengths (with no constraints on the $a_n$s for the moment) a sequence of graph--theoretic trees $(R_n)$ that codes its genealogy as follows: 
\begin{enumerate}
\item[$\bullet$] $R_1$ is the tree composed by a unique vertex, labeled ${\tiny \circled{1}}$ 
\item[$\bullet$] if in $T_n$ the branch $\mathsf b_n$ is glued on the branch $\mathsf b_i$, $i<n$, then $R_n$ is obtained from $R_{n-1}$ by grafting a new vertex, labeled ${\tiny \circled{n}}$, to the vertex ${\tiny \circled{i}}$.  
\end{enumerate}
The vertex ${\tiny \circled{1}}$ is considered as the root of $R_n, \forall n\geq 1$. This sequence of genealogical trees has been used by \cite{ADGO14} to study the boundedness of $\overline{\cup_{n\geq 1} T_n}$.

From now on it is assumed that $a_n=1$ for all $n\geq 1$. In that case, for all $n$, $R_n$ is obtained by grafting the new vertex ${\tiny \circled{n}}$ to one vertex chosen uniformly at random amongst the $n-1$ vertices of $R_{n-1}$.  Hence $R_n$ is a \emph{uniform recursive tree} with $n$ leaves. Let $d_{R_n}$ denote the graph distance on $R_n$. It is well--known that 
\begin{equation}
\label{cvRecursif}
\frac{d_{R_n}({\tiny \circled{n}}, {\tiny \circled{1}})} {\ln(n)} \overset{\mathrm{a.s.}}{\underset{n\rightarrow \infty}\longrightarrow } 1 \quad \quad \text{and} \quad \quad \max_{1\leq i \leq n}\frac{d_{R_n}({\tiny \circled{i}}, {\tiny \circled{1}})}{\ln(n)} \overset{\mathrm{a.s.}}\longrightarrow e,
\end{equation}
see \cite{Devroye87, Pittel94}. Next we add lengths to  the edges of the trees $R_n,n\geq 2$. By construction, there exists a sequence of i.i.d. uniform r.v. $U_i,i\geq 1$ such that in $\cup_{n\geq 1} T_n$,
$$
d(L_i,\mathsf{root})=\sum_{j=1}^k U_{i_j}+U_i+1 \quad \text{if} \quad \mathsf b_1 \rightarrow \mathsf b_{i_1} \rightarrow \ldots \rightarrow \mathsf b_{i_k} \rightarrow \mathsf b_{i} 
$$
where the sequence $\mathsf b_1 \rightarrow \mathsf b_{i_1} \rightarrow \ldots \rightarrow \mathsf b_{i_k} \rightarrow \mathsf b_{i} $ represents the segments involved in the path from the root to $L_i$ (recall from the Introduction that the leaves are labelled by order of insertion). For all $n$ and all $2 \leq  i \leq n$, we decide to allocate the length $U_i$ to the edge in $R_n$ between the vertex ${\tiny \circled{i}}$ and its parent.  We denote by $\mathcal R_n$ this new tree with edge--lengths and by $d_{\mathcal R_n}$ the corresponding metric, so that finally,
\begin{equation}
\label{EqLien}
d(L_i,\mathsf{root})=d_{\mathcal R_n}({\tiny \circled{i}}, {\tiny \circled{1}})+1, \quad \text{for all leaves }L_i \in T_n.
\end{equation}  
See Figure \ref{fig:AggregGene} for an illustration. 

\begin{figure}
\begin{center}
\includegraphics[width=6cm, height=14cm,angle=-90]{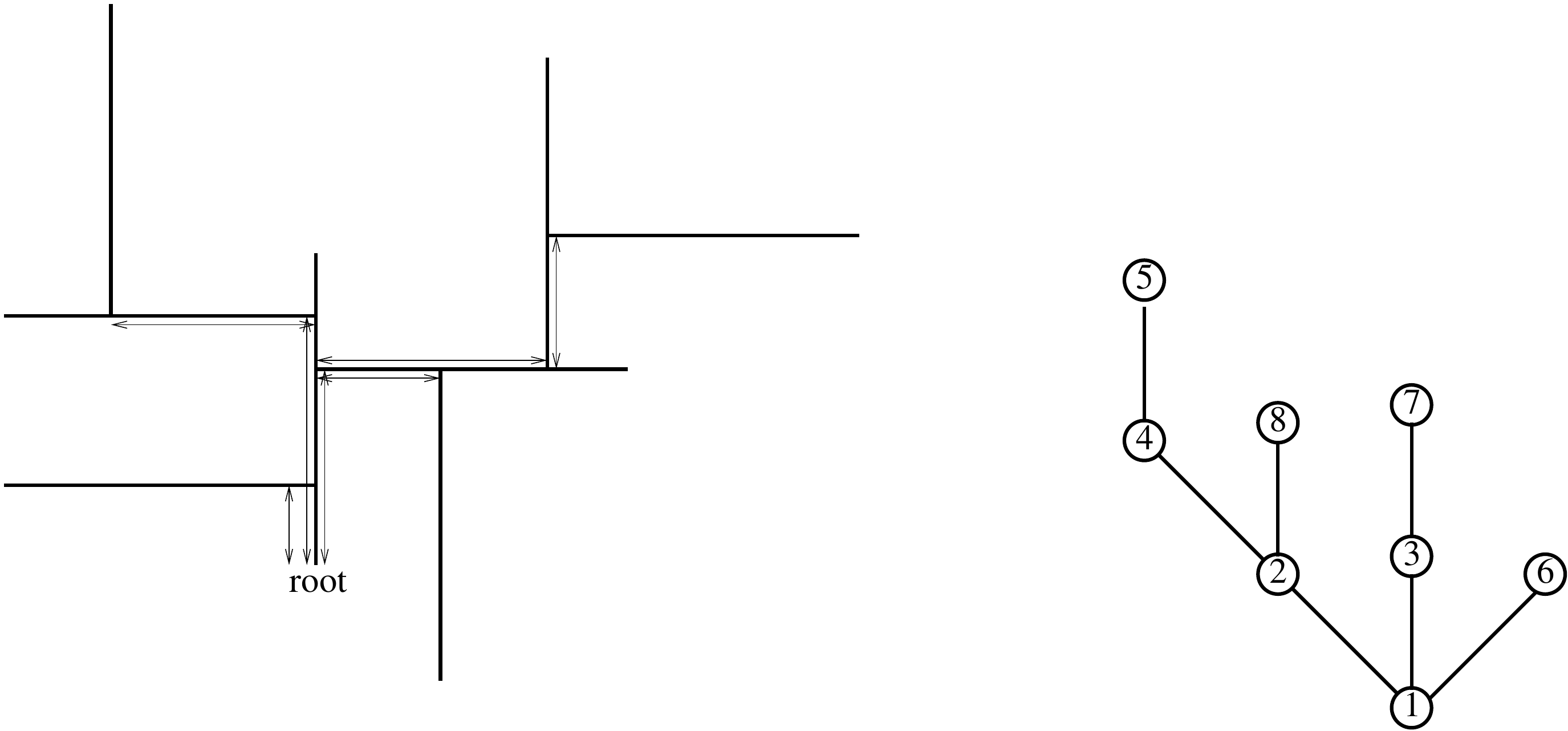}
\end{center}
\begin{picture}(0,0)(0,0)
\put(115,54){\small $U_6$}
\put(132,125){\small $L_1$}
\put(219,92){\small $L_2$}
\put(44,104){\small $L_3$}
\put(44,65){\small $L_6$}
\put(79,184){\small $L_7$}
\put(278,124){\small $L_5$}
\put(191,171){\small $L_4$}
\put(165,11){\small $L_8$}
\put(143,65){\small $U_2$}
\put(202,105){\small $U_5$}
\put(164,102){\small $U_4$}
\put(149,83){\small $U_8$}
\put(121,78){\small $U_3$}
\put(387,25){\small $U_2$}
\put(110,95){\small $U_7$}
\put(418,30){\small $U_3$}
\put(350,93){\small $U_5$}
\put(437,25){\small $U_6$}
\put(384,60){\small $U_8$}
\put(418,65){\small $U_7$}
\put(354,55){\small $U_4$}
\end{picture}
\caption{On the left, a version of the tree $T_8$. On the right, the associated genealogical tree with edge--lengths $\mathcal R_8$. The discrete (graph--theoretic) tree $R_8$ is obtained from $\mathcal R_8$ by forgetting the uniform lengths $U_i,2\leq i\leq 8$.} \label{fig:AggregGene}
\end{figure}
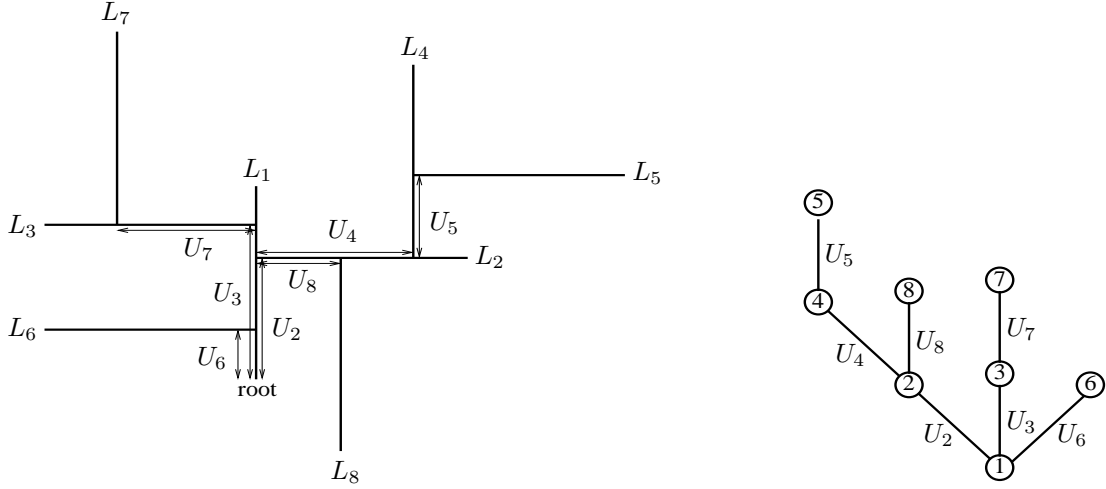

\bigskip

\textbf{Height of a typical vertex in $T_n$, height of leaf $L_n$, height of a uniform leaf of $T_n$.} The strong law of large numbers and the convergence on the left of (\ref{cvRecursif}) then clearly yield that $d_{\mathcal R_n}({\tiny \circled{n}}, {\tiny \circled{1}})/\ln(n)$ converges a.s. to $1/2$. This in turn yields that $d(L_n,\mathsf{root})/\ln(n)$ converges a.s. to $1/2$ and  that
$$
\frac{D_n}{\ln(n)}  \overset{\mathbb P}{\underset{n\rightarrow \infty}\longrightarrow } \frac{1}{2}
$$
since $\mathsf b_{n+1}$ is inserted on a uniform point of $T_{n}$. (More precisely, if we note, for each $n$, $\overline D_{n}$ the distance to the root of the insertion point of  $\mathsf b_{n+1}$ on $T_{n}$, we obtain versions of the $D_n$s that converge almost surely: $\overline D_n/\ln(n)\rightarrow 1/2$ a.s.). 

Moreover, from the (a.s.) convergence of $d(L_n,\mathsf{root})/\ln(n)$ to $1/2$, it is easy to get the convergence in probability of $d(L_{n,\star},\mathsf{root})/\ln(n)$ to $1/2$, where $L_{n,\star}$ is a uniform leaf of $T_n$. We let the reader adapt the proof seen in Section \ref{SecOneDim} for regularly varying sequences $(a_n)$ with a strictly positive index.

\bigskip

\textbf{Height of $T_n$.} From (\ref{EqLien}) it is clear that the height $H_n$ of $T_n$ has the same asymptotic behavior as the height of $\mathcal R_n$. Using results by Broutin and Devroye \cite{BD06} on the asymptotic behavior of heights of certain trees with edge--lengths, we obtain:

\bigskip

\begin{prop} 
As $n \rightarrow \infty$,
$$
\max_{1\leq i \leq n}\frac{d_{\mathcal R_n}({\emph{\tiny \circled{i}}}, \emph{{\tiny \circled{1}})}}{\ln(n)} \overset{\mathrm{\mathbb P}}\longrightarrow \frac{e^{\beta^*}}{2\beta^*},
$$
where $\beta^*$ is the unique solution in $(0,\infty)$ to the equation $2(e^{\beta}-1)=\beta e^{\beta}$. 
\end{prop}

\begin{proof}  We use several remarks or technics of \cite{BD06} and invite the reader to refer to this paper for details. First, according to the paragraph following  Theorem 3 in \cite{BD06}, the  random recursive trees $(\mathcal R_n)$ can be coupled with random binary trees with edge--lengths so as to fit the framework of \cite[Theorem 1]{BD06} on the asymptotic of heights of binary trees with edge--lengths. From this theorem, we then know that
$$
\max_{1\leq i \leq n}\frac{d_{\mathcal R_n}({\tiny \circled{i}}, {\text{\tiny \circled{1}}})}{\ln(n)} \overset{\mathbb P}{\underset{n \rightarrow \infty}\longrightarrow} c
$$ 
where $c$ is defined a few lines below. Let us first introduce some notation.

Let $E$ denote an exponential r.v. with parameter 1 and  $Z$ a real--valued r.v. with distribution \linebreak $(\delta_0(\mathrm dx)+\mathbbm 1_{[0,1]}(x) \mathrm dx)/2$, where $\delta_0$ denotes the Dirac measure at 0 and $\mathrm dx$ the Lebesgue measure on $\mathbb R$. 
Note that $\mathbb E[E]=1$ and $\mathbb E[Z]=1/4$.
Moreover, $$\Lambda_Z(t):=\ln\left(\mathbb E\left[e^{tZ}\right] \right)=\ln\left(1+\frac{e^t-1}{t} \right)-\ln(2), \quad \text{for }t \neq 0$$ and $\Lambda_Z(0)=0$. The corresponding Fenchel--Legendre transform $\Lambda^{*}_Z(t):=\sup_{\lambda \in \mathbb R}\left\{\lambda t-\Lambda_Z(\lambda)\right\}$ is then given by
$$
\Lambda^{*}_Z(t)=t \lambda(t)-\ln(h(\lambda(t)))+\ln(2) \quad \text{for } 0<t<1,
$$
and $\Lambda^{*}_Z(t)=+\infty$ for $t \notin(0,1)$,
where $h(u)=1+(e^u-1)/u,$ for $u \in \mathbb R$ ($h(0)=2$), and for $t \in (0,1)$, $\lambda(t)$ is defined by
$$t=\frac{h'(\lambda(t))}{h(\lambda(t))}$$
(the function $u \in \mathbb R \mapsto h'(u)/h(u) \in (0,1)$ -- with the convention $h'(0)/h(0)=1/4$ -- is bijective, increasing).
For the r.v. $E$, we more simply have
$$
\Lambda^{*}_E(t)=t-1-\ln(t) \quad \text{for }0<t<1
$$
and $\Lambda^{*}_E(t)=+\infty$ for $t \notin (0,1)$. According to \cite[Theorem 1]{BD06},  the limit $c$ introduced above is defined 
as the unique maximum of $\alpha/\rho$ along the curve
\begin{eqnarray}
\label{curve}
&&\left\{(\alpha,\rho):\Lambda^{*}_Z(\alpha)+\Lambda^{*}_E(\rho)=\ln(2), 0<\rho < 1, \frac{1}{4}\leq \alpha <1  \right\} \\
\nonumber
&=& \left\{(\alpha,\rho) : \alpha \lambda(\alpha)-\ln(h(\lambda(\alpha)))+\rho-1-\ln(\rho)=0 ), 0<\rho < 1, \frac{1}{4}\leq \alpha <1  \right\}
\end{eqnarray}
(according to \cite[Lemma 1]{BD06}, this curve is increasing and concave).

It remains to determine this maximum. 
We reason like  Broutin and Devroye at the end of their proof of \cite[Theorem 3]{BD06}. The slope of the curve is 
$$
\frac{\mathrm d \rho}{\mathrm d \alpha}=\frac{\lambda(\alpha)}{\frac{1}{\rho}-1}
$$
and on the other hand, at the maximum
$$
\frac{\mathrm d \rho}{\mathrm d \alpha}=\frac{\rho}{\alpha}.
$$
Hence, at the maximum
\begin{equation*}
\label{eqalpha}
\alpha_{\mathrm{max}} \lambda (\alpha_{\mathrm{max}})=1-\rho_{\mathrm{max}}.
\end{equation*}
Plugging this in (\ref{curve})  gives $\rho_{\mathrm{max}}=1/h(\lambda(\alpha_{\mathrm{max}}))$, which  gives in turn
$$
\alpha_{\mathrm{max}} \lambda (\alpha_{\mathrm{max}})=1-\frac{1}{h(\lambda(\alpha_{\mathrm{max}}))}.
$$
Setting $\beta_{\mathrm{max}}=\lambda(\alpha_{\mathrm{max}}) \Leftrightarrow \alpha_{\mathrm{max}}=h'(\beta_{\mathrm{max}})/h(\beta_{\mathrm{max}})$, this is equivalent to
$$
\frac{h'(\beta_{\mathrm{max}})}{h(\beta_{\mathrm{max}})} \beta_{\mathrm{max}}=1-\frac{1}{h(\beta_{\mathrm{max}})}.
$$ 
Simple manipulations then give
$$
2(e^{\beta_{\mathrm{max}}}-1)=\beta_{\mathrm{max}} e^{\beta_{\mathrm{max}}},
$$
which then leads to
$$
c=\frac{\alpha_{\mathrm{max}}}{\rho_{\mathrm{max}}}=\frac{1}{2}\frac{e^{\beta_{\mathrm{max}}}}{\beta_{\mathrm{max}}}.
$$
\end{proof}

\bigskip

\bigskip

\textbf{Acknowledgments.} I warmly thank Louigi Addario--Berry for pointing at the reference \cite{BD06}.

\bibliographystyle{siam}
\addcontentsline{toc}{section}{References}
\bibliography{FragAvril16}

\begin{thebibliography}{10}

\bibitem{Ald91a}
{\sc D.~Aldous}, {\em The continuum random tree. {I}}, Ann. Probab., 19 (1991),
  pp.~1--28.

\bibitem{Ald91}
\leavevmode\vrule height 2pt depth -1.6pt width 23pt, {\em The continuum random
  tree. {II}. {A}n overview}, in Stochastic analysis ({D}urham, 1990), vol.~167
  of London Math. Soc. Lecture Note Ser., Cambridge Univ. Press, Cambridge,
  1991, pp.~23--70.

\bibitem{Ald93}
{\sc D.~Aldous}, {\em The continuum random tree {III}}, Ann. Probab., 21
  (1993), pp.~248--289.

\bibitem{ADGO14}
{\sc O.~Amini, L.~Devroye, S.~Griffiths, and N.~Olver}, {\em Explosion and
  linear transit times in infinite trees}.
\newblock To appear in Probab. Theory Related Fields. Preprint --
  arXiv:1411.4426.

\bibitem{Bill99}
{\sc P.~Billingsley}, {\em Convergence of probability measures}, Wiley Series
  in Probability and Statistics: Probability and Statistics, John Wiley \&
  Sons, Inc., New York, second~ed., 1999.

\bibitem{BGT89}
{\sc N.~H. Bingham, C.~M. Goldie, and J.~L. Teugels}, {\em Regular variation},
  vol.~27 of Encyclopedia of Mathematics and its Applications, Cambridge
  University Press, Cambridge, 1989.

\bibitem{BD06}
{\sc N.~Broutin and L.~Devroye}, {\em Large deviations for the weighted height
  of an extended class of trees}, Algorithmica, 46 (2006), pp.~271--297.

\bibitem{CH15+}
{\sc N.~Curien and B.~Haas}, {\em Random trees constructed by aggregation}.
\newblock Preprint -- arXiv:1411.4255.

\bibitem{Devroye87}
{\sc L.~Devroye}, {\em Branching processes in the analysis of the heights of
  trees}, Acta Inform., 24 (1987), pp.~277--298.

\bibitem{DLG02}
{\sc T.~Duquesne and J.-F. Le~Gall}, {\em Random trees, {L}\'evy processes and
  spatial branching processes}, Ast\'erisque,  (2002), pp.~vi+147.

\bibitem{GH15}
{\sc C.~Goldschmidt and B.~Haas}, {\em A line-breaking construction of the
  stable trees}, Electron. J. Probab., 20 (2015), pp.~1--24.

\bibitem{LGLJ98}
{\sc J.-F. Le~Gall and Y.~Le~Jan}, {\em Branching processes in {L}\'evy
  processes: the exploration process}, Ann. Probab., 26 (1998), pp.~213--252.

\bibitem{Pittel94}
{\sc B.~Pittel}, {\em Note on the heights of random recursive trees and random
  {$m$}-ary search trees}, Random Structures Algorithms, 5 (1994),
  pp.~337--347.

\bibitem{Delphin16}
{\sc D.~S\'enizergues}, {\em Random gluing of $d$--dimensional metric spaces},
  In preparation.

\end{thebibliography}

\end{document}